\documentclass[11pt,a4paper]{article}
\usepackage{color}
\definecolor{deepblue}{rgb}{0,0,0.5}
\definecolor{deepred}{rgb}{0.6,0,0}
\definecolor{deepgreen}{rgb}{0,0.5,0}

\usepackage{amsmath,amssymb}
\usepackage{bm}
\usepackage{amsthm}
\usepackage{anyfontsize}
\usepackage{url}
\usepackage{pdflscape}
\usepackage{afterpage}
\usepackage{capt-of}
\usepackage{graphicx}
\usepackage{array}
\usepackage{varwidth}
\usepackage{units}
\usepackage{hyperref}
\usepackage{bbm}
\usepackage{geometry}
\usepackage{ esint }
\definecolor{ao(english)}{rgb}{0.0, 0.5, 0.0}
\usepackage{listings}
\usepackage{fancyhdr}
\usepackage{enumitem}
\usepackage[
    backend=bibtex,
   bibstyle=ieee,
    citestyle=numeric-comp,
    natbib=true,
    sortlocale=en_GB,
    sorting = nyt,
    url=true,
    doi=true,
    arxiv=true,
    eprint=true
]{biblatex}
\DeclareFieldFormat{eprint:arXiv}{%
  arXiv preprint: \href{https://arxiv.org/abs/#1}{#1 \texttt{[\printfield{eprintclass}]}}%
}

\usepackage{caption}
\usepackage{mathtools}
\usepackage{subfig}
\usepackage{appendix}
\usepackage{tikz-cd}
\hypersetup{
    colorlinks,
    linkcolor={blue!50!black},
    citecolor={blue!50!black},
    urlcolor={blue!80!black}
}

\newcommand{\im}{\mathrm{i}}
\newcommand{\E}{\mathrm{e}}

\def\Xint#1{\mathchoice
{\XXint\displaystyle\textstyle{#1}}%
{\XXint\textstyle\scriptstyle{#1}}%
{\XXint\scriptstyle\scriptscriptstyle{#1}}%
{\XXint\scriptscriptstyle\scriptscriptstyle{#1}}%
\!\int}
\def\XXint#1#2#3{{\setbox0=\hbox{$#1{#2#3}{\int}$}
\vcenter{\hbox{$#2#3$}}\kern-.5\wd0}}

\def\dashint{\Xint-}

\RequirePackage{algorithm}[0.1]

\RequirePackage[capitalize,nameinlink]{cleveref}[0.19]
\crefformat{equation}{\textup{#2(#1)#3}}
\crefrangeformat{equation}{\textup{#3(#1)#4--#5(#2)#6}}
\crefmultiformat{equation}{\textup{#2(#1)#3}}{ and \textup{#2(#1)#3}}
{, \textup{#2(#1)#3}}{, and \textup{#2(#1)#3}}
\crefrangemultiformat{equation}{\textup{#3(#1)#4--#5(#2)#6}}%
{ and \textup{#3(#1)#4--#5(#2)#6}}{, \textup{#3(#1)#4--#5(#2)#6}}{, and \textup{#3(#1)#4--#5(#2)#6}}

\Crefformat{equation}{#2Equation~\textup{(#1)}#3}
\Crefrangeformat{equation}{Equations~\textup{#3(#1)#4--#5(#2)#6}}
\Crefmultiformat{equation}{Equations~\textup{#2(#1)#3}}{ and \textup{#2(#1)#3}}
{, \textup{#2(#1)#3}}{, and \textup{#2(#1)#3}}
\Crefrangemultiformat{equation}{Equations~\textup{#3(#1)#4--#5(#2)#6}}%
{ and \textup{#3(#1)#4--#5(#2)#6}}{, \textup{#3(#1)#4--#5(#2)#6}}{, and \textup{#3(#1)#4--#5(#2)#6}}

\newmuskip\pFqmuskip

\newcommand*\pFq[6][8]{%
  \begingroup 
  \pFqmuskip=#1mu\relax
  \mathchardef\normalcomma=\mathcode`,
  \mathcode`\,=\string"8000
  \begingroup\lccode`\~=`\,
  \lowercase{\endgroup\let~}\pFqcomma
  {}_{#2}F_{#3}{\left(\genfrac..{0pt}{}{#4}{#5};#6\right)}%
  \endgroup
}
\newcommand{\pFqcomma}{{\normalcomma}\mskip\pFqmuskip}
\DeclareFieldFormat{postnote}{#1}
\usepackage{algpseudocode}
\algnewcommand{\Initialize}[1]{
  \State \textbf{Initialize:}
  \Statex \hspace*{\algorithmicindent}\parbox[t]{.8\linewidth}{\raggedright #1}
}
\algnewcommand{\Indent}[2]{
  \State {#1}
  \vspace{-2mm}
  \Statex \hspace*{\algorithmicindent}\parbox[t]{.9\linewidth}{\raggedright #2}
}

\usepackage{todonotes}

\newtheorem{theorem}{Theorem}[section]
\newtheorem{lemma}{Lemma}[section]
\newtheorem{definition}{Definition}[section]
\newtheorem{corollary}{Corollary}[section]
\newtheorem{remark}{Remark}[section]
\numberwithin{equation}{section}

\title{Explicit fractional Laplacians\\ and Riesz potentials of classical functions}

\author{Timon S.~Gutleb \thanks{School of Computer Science, University of Leeds, UK, {\tt t.s.gutleb@leeds.ac.uk}} \and Ioannis P.~A.~Papadopoulos\thanks{Mathematical Institute, University of Oxford, UK,  \tt{ioannis.papadopoulos@maths.ox.ac.uk}.}}
\date{}

\begin{document}

\maketitle
\thispagestyle{empty}
\pagestyle{fancy}

\begin{abstract}
We prove and collect numerous explicit and computable results for the fractional Laplacian $(-\Delta)^s f(x)$ with $s>0$ as well as its whole space inverse, the Riesz potential, $(-\Delta)^{-s}f(x)$ with $s\in\left(0,\frac{1}{2}\right)$, subject to row-specific parameter and integrability conditions. Choices of $f(x)$ include weighted classical orthogonal polynomials such as the Legendre, Chebyshev, Jacobi, Laguerre and Hermite polynomials, or first and second kind Bessel functions with or without sinusoid weights. Some higher dimensional fractional Laplacians and Riesz potentials of generalized Zernike polynomials on the unit ball and its complement as well as whole space generalized Laguerre polynomials are also discussed. The aim of this paper is to aid in the continued development of numerical methods for problems involving the fractional Laplacian or the Riesz potential in bounded and unbounded domains -- both directly by providing useful basis or frame functions for spectral method approaches and indirectly by providing accessible ways to construct computable synthetic problems on which to test new numerical methods.
\end{abstract}

\section{Introduction}
Fractional calculus has been a topic of increasing research interest in both pure and applied mathematics in the past decades. Many natural science and engineering applications use fractional differential equations and the fractional Laplacian specifically, among them e.g.~models for acoustic waves in viscoelastic media with important applications in ultrasound-based cancer treatments and photoacoustic imaging cf.~\cite{treeby2018rapid,treeby2014modelling,treeby2012modeling,treeby2010k,gutleb2023static,king_modelling_2024}. The Riesz potential is an operator closely related to the fractional Laplacian which similarly enjoys numerous applications cf.~\cite{Jcarrilloradial}.  Some of the most successful and widely used numerical methods for fractional differential equations rely on approximation with simpler functions (e.g.~spectral methods). In light of the rising interest in numerical methods and simulations involving these nonlocal operators, the aim of this paper is to give explicit forms for the fractional Laplacian and Riesz potential of widely used basis functions in one dimension as well as certain geometries in arbitrary dimensions which are useful in numerical practice. These explicit expressions may be used as part of a spectral method to solve fractional differential equations numerically, cf.~\cref{sec:numerical} and \cite{papadopoulos2023frame}, or to construct non-trivial toy problems to test novel numerical approaches against. We give brief descriptions and detailed references for all of the classical functions discussed. A companion Julia package, \href{https://github.com/TSGut/ExplicitFractionalLaplacianRieszExamples.jl}{\texttt{ExplicitFractionalLaplacianRieszExamples.jl}}~\cite{explicitfractionallaplacianrieszexamples-zenodo}, provides deterministic numerical tests of the displayed formulas using \texttt{MeijerG.jl}~\cite{meijergjl-zenodo}.

The most common workflow for computing the explicit expressions is depicted in \cref{fig:workflow}. The goal is to write the simpler function as a Meijer-G function, apply the fractional Laplacian via \cref{thm:meijergDKKtheorem1}, and then rewrite the explicit resultant expression in terms of simpler functions once more.
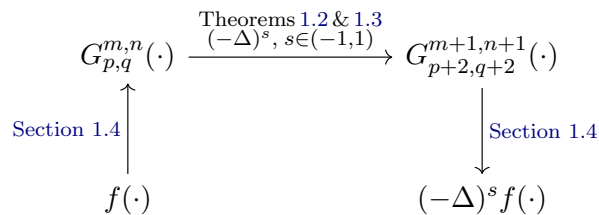
\begin{figure}[h!]
\centering
\begin{tikzcd}[column sep=7em, row sep=3em]
G^{m,n}_{p,q}(\cdot)  \arrow{r}{\substack{\text{Theorems} \, \labelcref{thm:meijergDKKtheorem1} \, \& \, \labelcref{thm:meijergDKKtheorem2}\\(-\Delta)^s}} &  G^{m+1,n+1}_{p+2,q+2}(\cdot)  \arrow{d}{\text{\cref{sec:meijergforms}}} \\%
f(\cdot) \arrow{u}{\text{\cref{sec:meijergforms}}} & (-\Delta)^s f(\cdot)
\end{tikzcd}
\caption{The primary workflow for computing explicit fractional Laplacians and Riesz potentials of classical functions.} \label{fig:workflow}
\end{figure}

With numerical applications and implementations in mind, wherever possible, we will forgo notational abstractions which are often introduced in the literature for the purpose of making results appear more compact, e.g.~compressing repeating expressions into a new variable. All of our explicit formulae in one dimension will thus only depend on the order $s$ of the fractional Laplacian (or Riesz potential), the degree or order $n$ (or $\nu$) of the weighted classical polynomials and Bessel functions of first and second kind, as well as any basis parameters that are inherently present such as $a$ and $b$ for the Jacobi polynomials $P_n^{(a,b)}(x)$. While this choice makes some expressions longer, we believe it ultimately makes them less obfuscated, easier to implement and thus more useful in applications.

The sections of this paper are organized as follows: In \cref{sec:meijergintro} we introduce the Meijer-G function and discuss its properties which are relevant to our use case. In \cref{sec:fractionallapandriesz} we introduce definitions for the fractional Laplacian and Riesz potential. \cref{sec:briefhistory} gives a brief historical overview of explicit form solutions for the fractional Laplacian and Riesz potential and introduces the main underlying theorem due to Dyda, Kuznetsov and Kwa\`snicki \cite{dyda2017fractional}. In Sections \ref{sec:classicalintro} and \ref{sec:besselintro} we introduce the classical orthogonal polynomial bases and Bessel functions, respectively, for which we provide explicit form fractional Laplacians and Riesz potentials. \cref{sec:meijergforms} shows how to cast the above-mentioned functions into the form of a Meijer-G function which is then used in \cref{sec:fractionalexplicitforms1D} to give explicit fractional Laplacian and Riesz potential formulas in one dimension. \cref{sec:fractionalexplicitformshighD} provides related explicit expressions in higher dimensional geometries. Finally, \cref{sec:numerical} gives some simple examples on how to apply these formulas to solve certain fractional differential equations, with references to more complicated examples.

\subsection{Hypergeometric functions and Meijer-G functions}\label{sec:meijergintro}
The Meijer-G function is one of many available generalizations of hypergeometric functions, which themselves are one of many available generalizations of classical orthogonal polynomials. The Meijer-G includes a vast range of functions as special cases, many of which may be found in the excellent formula collections due to Prudnikov, Byrchkov and Marichev \cite[Ch.~8.4]{prudnikovVol3}. In this section we provide definitions as well as general characteristics and tools for working with these functions and end with a brief overview of recent numerical methods for computing hypergeometric functions. For an introduction to the world of Meijer-G functions we also refer to \cite{beals2013meijer}.
\begin{definition}[Mellin--Barnes integral]
A Mellin--Barnes integral is a complex contour integral whose integrands are products of gamma functions with reciprocals of gamma functions multiplied by $z^{s}$, that is integrals of the form:
\begin{align*}
\frac{1}{2\pi \im} \int_C \frac{\prod_{r=1}^{m} \Gamma(b_r-\beta_r s) \prod_{r=1}^n \Gamma(1+\alpha_r s-a_r)}{\prod_{r=m+1}^q \Gamma(1+\beta_r s-b_r) \prod_{r=n+1}^p \Gamma(a_r-\alpha_rs)}z^s \mathrm{d}s,
\end{align*}
where $m,n,p,q \in \mathbb{N}_0$ , $0 \leq m \leq q$, $0\leq n\leq p$, $\alpha_r, \beta_r > 0$, $\Gamma(\cdot)$ denotes the $\Gamma$-function, and the contour $C$ is chosen such that the poles of the terms $\Gamma(b_r - \beta_r s)$ with $1\leq r\leq m$ are separated from those of $\Gamma(1+\alpha_r s -a_r)$ with $1\leq r \leq n$, cf.~\cite[2.4]{paris2001asymptotics}. 
\end{definition}
\begin{definition}[Meijer-G function]
Let $a_1,a_2,...,a_p$ and $b_1,b_2,...,b_q$ be real or complex, $m,n,p,q \in \mathbb{N}_0$, $m \in [0,q]$ and $n \in [0,p]$ and let none of $a_k-b_j$ be positive integers for $1\leq k \leq n$ and $1\leq j \leq m$. Meijer-G functions are defined to be Mellin--Barnes integrals with $\beta_r = \alpha_r = 1$, i.e.
\begin{align*}
{G^{m,n}_{p,q}}\left(z \left| \begin{matrix}
\mathbf{a} \\ \mathbf{b}
\end{matrix}\right.\right)&={G^{m,n}_{p,q}}\left(z \left| {a_%
{1},\dots,a_{p}\atop b_{1},\dots,b_{q}}\right. \right)\\&\coloneqq \frac{1}{2\pi \im} \int_C \frac{\prod_{r=1}^{m} \Gamma(b_r-s) \prod_{r=1}^n \Gamma(1+ s-a_r)}{\prod_{r=m+1}^q \Gamma(1+ s-b_r) \prod_{r=n+1}^p \Gamma(a_r-s)}z^s \mathrm{d}s.
\end{align*}
\end{definition}
The notation $(\cdot)_k$ in the below definition of ${{}_{p}F_{q}}$ will reappear frequently in this paper and denotes the Pochhammer symbol, cf.~\cite[5.2(iii)]{nist_2018}. As multiple objects exist in the literature with similar notation and names, we include a definition here.
\begin{definition}[Pochhammer symbol] The Pochhammer symbol (or `rising factorial') is defined to be
\begin{align*}
(z)_k \coloneqq \prod_{r=0}^{k-1} (z+r).
\end{align*}
\end{definition}
\begin{remark}
If $z$ is not a non-positive integer, i.e.~$(-z) \notin \mathbb{N}_0$, the Pochhammer symbol can be computed via
\begin{align*}
(z)_k = \frac{\Gamma(z+k)}{\Gamma(z)}.
\end{align*}
\end{remark}
\begin{definition}[Hypergeometric ${}_pF_q$]
We define the ${}_pF_q$ hypergeometric function to be
\begin{align*}
{{}_{p}F_{q}}\left(\left.{a_{1},\dots,a_{p}\atop b_{1},\dots,b_{q}}\right|z\right)\coloneqq\sum_{k%
=0}^{\infty}\frac{{\left(a_{1}\right)_{k}}\cdots{\left(a_{p}\right)_{k}}}{{%
\left(b_{1}\right)_{k}}\cdots{\left(b_{q}\right)_{k}}}\frac{z^{k}}{k!},
\end{align*}
whenever this series converges.
\end{definition}

We will not concern ourselves with general case existence or convergence questions for these definitions in this paper, as this has been discussed at length elsewhere. For instance the ${}_pFq$ hypergeometric series are  known to be absolutely convergent for all $z \in \mathbb{C}$ if $p < q+1$ and for $|z|<1$ if $p=q+1$ \cite{bateman1953higher,kilbas2016generalized}. The convergence conditions of Mellin--Barnes integrals are discussed in detail in \cite{paris2001asymptotics}, see also \cite[16.17]{nist_2018} for the Meijer-G function case.

Computing general case Meijer-G functions is typically avoided in practice; instead one uses polyalgorithms which use direct reductions to hypergeometric functions or domain decomposition followed by reduction to hypergeometric functions. While many of the formulae we present in this paper are of Meijer-G form, we always also include reduced classical or hypergeometric forms as these are much easier to compute in practice. The intimate relationship between ${}_pF_q$ hypergeometric functions and Meijer-G functions has been known for decades \cite{prudnikovVol3}, with a recent paper by Kilbas et al.~\cite{kilbas2016generalized} exploring their connection in detail. The following are two equivalent ways to express their relationship \cite[16.18.1]{nist_2018}:
\begin{align*}
{{}_{p}F_{q}}\left(\left. {a_{1},\dots,a_{p}\atop b_{1},\dots,b_{q}}\right| z\right)
&=\left({\frac{\prod\limits_{k=1}^{q}\Gamma\left(b_{k}\right)}{\prod\limits_{k=1}^{p}\Gamma\left(a_{k}\right)}}\right){G^{1,p}_{p,q+1}}\left(-z\left|{1-a_{1},\dots,1-a_{p}\atop 0,1-b_{1},\dots,1-b_{q}}\right.\right)\\
&=\left({\frac{\prod\limits_{k=1}^{q}\Gamma\left(b_{k}\right)}{\prod\limits_{k=1}^{p}\Gamma\left(a_{k}\right)}}\right){G^{p,1}_{q+1,p}}\left(-\frac{1}{z}\left| {1,b_{1},\dots,b_{q}\atop a_{1},\dots,a_{p}}\right.\right).
\end{align*}
While computing with Meijer-G functions directly is usually avoided, they are extremely useful for the derivation of broad scope theorems not only due to their extreme generality but also because they satisfy numerous useful mathematical properties.

Both hypergeometric ${}_p F_q$ functions and Meijer-G functions are invariant with respect to permutations of upper or lower parameters respectively, which follows from their definitions due to commutativity. Furthermore, the following properties hold \cite[16.19.1--16.19.2]{nist_2018}:
\begin{align}\label{eq:meijerGsym1}
    {G^{m,n}_{p,q}}\left(\frac{1}{z}\left|{a_{1},\dots,a_{p}\atop b_{1},\dots,b_{q}}\right.\right)={G^{n,m}_{q,p}}\left(z\left|{1-b_{1},\dots,1-b_{q}\atop 1-a_{1},\dots,1-a_{p}}\right.\right),\\
z^{\mu}{G^{m,n}_{p,q}}\left(z\left|{a_{1},\dots,a_{p}\atop b_{1},\dots,b_{q}}\right.\right)={G^{m,n}_{p,q}}\left(z\left|{a_{1}+\mu,\dots,a_{p}+\mu\atop b_{1}+\mu,\dots,b_{q}%
+\mu}\right.\right). \label{eq:meijerGsym2}
\end{align}
We refer to \eqref{eq:meijerGsym1} as the \emph{argument inversion} property and to \eqref{eq:meijerGsym2} as the \emph{multiplicative shift} property of Meijer-G functions. A glance at their definition also shows that the Meijer-G functions further satisfy the \emph{cancellation-reduction} properties (cf. \cite[8.2.2.8]{prudnikovVol3}):
\begin{subequations}
\begin{align}\label{eq:meijerGfullcancellation}
    {G^{m,n}_{p,q}}\left(\left.{c, \mathbf{a}\atop \mathbf{b},c}
    \right|z\right)={G^{m,n-1}_{p-1,q-1}}\left(\left.{\mathbf{a}\atop \mathbf{b}}    \right| z\right),\\
        {G^{m,n}_{p,q}}\left(\left.{\mathbf{a}, c\atop c, \mathbf{b}}
    \right|z\right)={G^{m-1,n}_{p-1,q-1}}\left(\left.{\mathbf{a}\atop \mathbf{b}}    \right| z\right).\label{eq:meijerGfullcancellation2}
\end{align}
\end{subequations}
Moreover, the Meijer-G function is symmetric with respect to the order of the parameters within each of the four sets $\{a_1, \dots, a_n \}$, $\{a_{n+1},\dots,a_p\}$, $\{b_1,\dots,b_m\}$, and \{$b_{m+1}, \dots, b_q\}$ \cite[8.2.2.1]{prudnikovVol3}, e.g.
\begin{align}
G^{2,2}_{3,3}\left(z \left | a_1, a_2, a_3 \atop b_1, b_2, b_3 \right. \right) 
= G^{2,2}_{3,3}\left(z \left | a_2, a_1, a_3 \atop b_1, b_2, b_3 \right. \right) 
=  G^{2,2}_{3,3}\left(z \left | a_1, a_2, a_3 \atop b_2, b_1, b_3 \right. \right).
\end{align}
This is known as the \emph{permutation symmetry} property.  

In addition to the above-mentioned relationship with the ${}_p F_q$ hypergeometric function, the Meijer-G function with certain special case values $(m,n,p,q)$ satisfies reduction formulae to certain hypergeometric function expressions. There are a large number of such special cases, many of which are collected in \cite{prudnikovVol3} and \cite{WolframFunctions2022}. We reproduce some such formulae which we will use in our proofs in \cref{sec:app:reduction}.

\subsection{Classical orthogonal polynomials}\label{sec:classicalintro}
In general, a complete ordered set of univariate polynomials $\{p_n\}_{n\in\mathbb{N}_0}$ is said to be orthogonal with respect to a weight function $w(x)$ on a characteristic domain $\Omega \subset \mathbb{R}$ if the polynomials satisfy
\begin{align*}
\int_\Omega w(x) p_n(x) p_m(x) \mathrm{d}x = c_n \delta_{n,m},
\end{align*}
where $c_n$ is independent of $x$ and $\delta_{n,m}$ denotes the Kronecker delta.

All orthogonal polynomials in one dimension have a three-term recurrence:
\begin{align*}
p_{k+1}(x)=(A_{k}x+B_{k})p_{k}(x)-C_{k}p_{k-1}(x),
\end{align*}
where $A_k$, $B_k$ and $C_k$ are constants independent of $x$. Explicit values of $A_k$, $B_k$ and $C_k$ for the classical orthogonal polynomials discussed in this paper can be found in \cite[18.9(i)]{nist_2018}. Orthogonal polynomials have enjoyed frequent and successful use in various numerical schemes such as integration via quadrature methods or solving ODEs, PDEs and integral equations via spectral methods \cite{olver2020fast}.

In what follows we provide definitions as well as succinct characterizations of all three types of classical orthogonal polynomials. We present each basis of orthogonal polynomials in their classical non-orthonormal normalization and to remove any potential ambiguity we also present multiple explicit formulae to compute them in their canonical form by means of series expansions or hypergeometric functions. We only discuss classical orthogonal polynomials in this paper, for numerical methods to work with non-classical orthogonal polynomials we refer to \cite{gutleb2023polynomial}.
\subsubsection{Jacobi, Legendre, Chebyshev and Gegenbauer polynomials}
The Jacobi polynomials $\{ P^{(\alpha,\beta)}_n(x) \}_{n\in \mathbb{N}}$ are a family of complete univariate bases of classical orthogonal polynomials on $(-1,1)$ with basis parameters $\alpha,\beta \in \mathbb{R}$ such that $\alpha,\beta >-1$. They are orthogonal with respect to the weight function
\begin{align*}
w^{(\alpha,\beta)}_P(x) = (1-x)_+^\alpha (1+x)_+^\beta.
\end{align*}
Throughout this paper, for $a\in\mathbb{R}$, we use the zero-extended power
\begin{align*}
(z)_+^a \coloneqq \begin{cases} z^a, & z>0,\\ 0, & z\leq 0. \end{cases}
\end{align*}
In their usual normalization \cite[18.3, 18.5.7]{nist_2018} the Jacobi polynomials have the explicit representations
\begin{align*}
P^{(\alpha,\beta)}_{n}\left(x\right)&=\sum_{\ell=0}^{n}\frac{{\left(n+\alpha+%
\beta+1\right)_{\ell}}{\left(\alpha+\ell+1\right)_{n-\ell}}}{\ell!\;(n-\ell)!}%
\left(\frac{x-1}{2}\right)^{\ell}\\ &=\frac{{\left(\alpha+1\right)_{n}}}{n!}{{}_{2%
}F_{1}}\left({-n,n+\alpha+\beta+1\atop\alpha+1} \left\lvert \frac{1-x}{2} \right. \right).
\end{align*}
In computing applications, especially in the context of spectral methods, one frequently encounters the problem of converting between two different parameter Jacobi polynomial bases \cite{olver2020fast}. Ways to improve the computational cost of such conversions were recently discussed in \cite{townsend2018fast}. The following general case conversion formula involving the ${}_3F_2$ hypergeometric function is sometimes useful \cite[05.06.17.0013.01]{WolframFunctions2022}:
\begin{align}
&P_n^{(a,b)}(x) = \sum _{k=0}^n c^{\alpha,\beta,k}_{a,b,n} P_k^{(\alpha ,\beta )}(x),
\end{align}
with coefficients given by
\begin{align}\label{eq:connectioncoeffjacobi}
\begin{split}
c^{\alpha,\beta,k}_{a,b,n}
&=\tfrac{(a+k+1)_{n-k} \Gamma (k+\alpha +\beta +1) (a+b+n+1)_k }{\Gamma (-k+n+1) \Gamma (2 k+\alpha +\beta +1)}\\
&\indent \times {}_3F_2\left(\left. \begin{array}{c} k-n, a+b+k+n+1, k+\alpha +1 \\ a+k+1,2 k+\alpha +\beta +2 \end{array} \right| 1\right).
\end{split}
\end{align}
When discussing Zernike polynomials and generalized orthogonal polynomials on arbitrary dimensional balls we will be making use of certain radially shifted Jacobi polynomials which generally take the form $P^{(\alpha,\beta)}_n(2|\mathbf{x}|^2-1)$, where $\mathbf{x} \in B_1 \subset \mathbb{R}^d$ is in the $d$ dimensional unit ball $B_1$. We discuss them in further detail in \cref{sec:fractionalexplicitformshighD}.

The Gegenbauer polynomials, also known under the name \emph{ultraspherical} polynomials \cite[18.3]{nist_2018} correspond to the $a=b=\lambda-1/2$ special case of the Jacobi polynomials with $\lambda > -1/2$ and $\lambda\neq0$, and are thus orthogonal with respect to the weight function $(1-x^2)^{\lambda-\frac{1}{2}}_+$. The excluded value $\lambda=0$ is degenerate in the Gegenbauer normalization used here. The notational discrepancy between the basis parameters of Gegenbauer and Jacobi polynomials is historical in nature. The Gegenbauer polynomials furthermore come with a different canonical normalization, such that the following relationships hold (cf.~\cite[18.7.1,
18.7.2]{nist_2018}):
\begin{align}
C^{(\lambda)}_{n}\left(x\right)&=\frac{{\left(2\lambda\right)_{n}}}{{\left(%
\lambda+\frac{1}{2}\right)_{n}}}P^{(\lambda-\frac{1}{2},\lambda-\frac{1}{2})}_%
{n}\left(x\right),\;\;
\text{and} \;\;
P^{(\alpha,\alpha)}_{n}\left(x\right)&=\frac{{\left(\alpha+1\right)_{n}}}{{%
\left(2\alpha+1\right)_{n}}}C^{(\alpha+\frac{1}{2})}_{n}\left(x\right).
\end{align}
The Chebyshev polynomials of the first and second kind are special cases of both the Jacobi and Gegenbauer polynomials with a different canonical normalization \cite[18.7.3,18.7.4]{nist_2018}
\begin{align*}
T_{n}\left(x\right)=\frac{1}{P^{(-\frac{1}{2},-\frac{1}{2})}_{n}\left(1\right)}P^{(-\frac{1}{2},-\frac{1}{2})}_{n}\left(x\right),
\;\; \text{and} \;\;
U_{n}\left(x\right)=C^{(1)}_{n}\left(x\right)=\frac{(n+1)}{P^{(\frac{1}{2},\frac{1}{2})}_{n}\left(1%
\right)} P^{(\frac{1}{2},\frac{1}{2})}_{n}\left(x\right).
\end{align*}
The Chebyshev polynomials of the first kind, conventionally denoted $\{ T_n(x) \}_{n\in \mathbb{N}}$ correspond to the $a=b=-1/2$ special case and are thus orthogonal with respect to the weight function $w^{(-\frac{1}{2},-\frac{1}{2})}_P(x)=\frac{1}{\sqrt{1-x^2}}$ on $(-1,1)$. Whereas the Chebyshev polynomials of the second kind $\{ U_n(x) \}_{n\in \mathbb{N}}$ correspond to the special case  $a=b=1/2$ and are orthogonal with respect to the weight function $w^{(\frac{1}{2},\frac{1}{2})}_P(x)=\sqrt{1-x^2}$ on $(-1,1)$.

\subsubsection{Laguerre polynomials}
The generalized Laguerre polynomials $\{ L^{\alpha}_n(x) \}_{n\in \mathbb{N}}$ are a family of complete univariate bases of classical orthogonal polynomials on the halfline $(0,\infty)$ with basis parameter $\alpha>-1$ and are orthogonal with respect to the weight function
\begin{align*}
w^\alpha_L(x) = \E^{-x} x^{\alpha}.
\end{align*}
In their usual normalization they have the explicit representations \cite[18.5.12]{nist_2018}
\begin{align}
\begin{split}
L_n^{(\alpha)} (x) = \sum_{k=0}^n \frac{(-1)^k}{k!} {n+\alpha \choose n-k} x^k
&=\sum_{k=0}^{n}\frac{(-1)^k{\left(\alpha+k+1%
\right)_{n-k}}}{(n-k)!\;k!}x^{k}\\
&= \frac{{\left(\alpha+1\right)_{%
n}}}{n!}{{}_{1}F_{1}}\left(\left.{-n\atop\alpha+1}\right\lvert x\right).
\end{split}
\end{align}
The Laguerre polynomials $\{ L_n(x) \}_{n\in \mathbb{N}}$ correspond to the $\alpha = 0$ special case of the generalized Laguerre polynomials, i.e. $L_n(x) \coloneqq L^0_n(x)$, and are thus orthogonal with respect to the weight function 
\begin{align*}
w^0_L(x) = \E^{-x}.
\end{align*}
As a result, an explicit series representation is obtained directly from the generalized Laguerre polynomials:
\begin{align}
L_n (x) = \sum_{k=0}^n \frac{(-1)^k}{k!} {n \choose n-k} x^k
=\sum_{k=0}^{n}\frac{(-1)^k{\left(k+1%
\right)_{n-k}}}{(n-k)!\;k!}x^{k}\
= {{}_{1}F_{1}}\left(\left.{-n\atop 1}\right\lvert x\right).
\end{align}
\subsubsection{Hermite polynomials}
The Hermite polynomials $\{ H_n(x) \}_{n\in \mathbb{N}}$ form a complete univariate basis of classical orthogonal polynomials on $\mathbb{R}$ and are orthogonal with respect to the weight function
\begin{align*}
w_H(x) = \E^{-x^2}.
\end{align*}
This basis is sometimes also referred to as the physicist's Hermite polynomials to distinguish them from the probabilist's Hermite polynomials usually denoted $\{ He_n(x) \}_{n\in \mathbb{N}}$ which are orthogonal with respect to $w_{He}(x) = \E^{-\frac{x^2}{2}}$ \cite{oldham2009atlas}. Throughout this paper we will use the physicist's Hermite polynomials and simply refer to them as Hermite polynomials -- equivalent expressions for the probabilist's definition can straightforwardly be obtained via scaling, since cf.~\cite[24:1:1]{oldham2009atlas}:
\begin{align*}
He_n(x) = \left(\sqrt{2}\right)^{-n} H_n\left( \frac{x}{\sqrt{2}} \right).
\end{align*}
The Hermite polynomials may be defined in many ways, a common definition being
\begin{align*}
H_n(x) \coloneqq (-1)^n \E^{x^2}\frac{\mathrm{d}^n}{\mathrm{d}x^n}\E^{-x^2}.
\end{align*}
In their usual normalization they have the explicit representations:
\begin{align*}
H_n(x) &= n! \sum_{m=0}^{\left\lfloor \tfrac{n}{2} \right\rfloor} \frac{(-1)^m}{m!(n - 2m)!} (2x)^{n - 2m}\\
&= \sqrt{\pi } 2^n \left(\frac{\, _1F_1\left(-\frac{n}{2};\frac{1}{2};x^2\right)}{\Gamma \left(\frac{1-n}{2}\right)}-\frac{2 x \, _1F_1\left(\frac{1-n}{2};\frac{3}{2};x^2\right)}{\Gamma \left(-\frac{n}{2}\right)}\right)\\
&= (2 x)^n \, _2F_0\left(-\frac{n}{2},\frac{1-n}{2};;-\frac{1}{x^2}\right),
\end{align*}
where $\left\lfloor \cdot \right\rfloor$ denotes the floor function.

\subsection{Bessel functions of the first and second kind}\label{sec:besselintro}
Bessel functions $w_\nu(z)$ arise as the solutions of Bessel's differential equation \cite[10.2.1]{nist_2018}
\begin{align}\label{eq:besseldiffeq}
z^{2}\frac{{\mathrm{d}}^{2}w_\nu(z)}{{\mathrm{d}z}^{2}}+z\frac{\mathrm{d}w_\nu(z)}{\mathrm{d%
}z}+(z^{2}-\nu^{2})w_\nu=0,
\end{align}
where $\nu \in \mathbb{C}$ is called the order of the respective Bessel function. These functions have a long history of successful applications, in particular where cylindrical and spherical coordinates are natural, ranging from finding solutions to classical equations such as the Laplace and Helmholtz equations to various physical problems such as Schr\"odinger's equation \cite{sasaki2016one,jiang2008efficient,kravchenko2017representation}, analysis of microtremors in geophysics \cite{chavez2005alternative,okada2006theory} and DNA X-ray diffraction \cite{kittel1968x}. For more applications we refer to \cite{korenev2002bessel} and \cite[10.72, 10.73]{nist_2018}.\\
The following solutions of \cref{eq:besseldiffeq} are called Bessel functions \emph{of first kind} \cite[
10.2.2]{nist_2018}:
\begin{align}
J_{\nu}\left(z\right)=(\tfrac{1}{2}z)^{\nu}\sum_{k=0}^{\infty}(-1)^{k}\frac{(%
\tfrac{1}{4}z^{2})^{k}}{k!\Gamma\left(\nu+k+1\right)}
= \frac{\left(\frac{z}{2}\right)^\nu}{\Gamma(\nu+1)} \;_0F_1 \left(\nu+1 \left\lvert -\frac{z^2}{4}\right.\right).
\end{align}
The following are called Bessel functions \emph{of second kind} \cite[10.2.3]{nist_2018}, alternatively Weber or Neumann functions, \cite{prudnikovVol3}:
 \begin{align}
Y_{\nu}\left(z\right)=\frac{J_{\nu}\left(z\right)\cos\left(\nu\pi\right)-J_{-%
\nu}\left(z\right)}{\sin\left(\nu\pi\right)}.
\end{align}
Note that for integer orders $\nu = n$, one defines $Y_{n}$ via $Y_{n} = \lim_{\nu\rightarrow n}Y_{\nu}$. References to third kind Bessel functions exist in the literature, where `third kind Bessel function' is an umbrella term for more commonly known Hankel functions of the first and second kind:
\begin{align}
H_\nu^{(1)}(x) &= J_\nu(x) + \im Y_\nu(x), \\
H_\nu^{(2)}(x) &= J_\nu(x) - \im Y_\nu(x).
\end{align}
We will discuss the fractional Laplacian and Riesz potential of Bessel functions of first and second kind in this manuscript along with some weighted variations. In some cases corollary results for third kind functions are straightforwardly derived from our results via the above relationship.

Lastly, we note that while the Bessel functions are generally not polynomials and thus naturally also not orthogonal polynomials, they are nevertheless useful for function approximation and satisfy various orthogonality conditions \cite[10.22]{nist_2018} such as
\begin{align*}
\int_{0}^{\infty}t^{-1}J_{\nu+2\ell+1}\left(t\right)J_{\nu+2m+1}\left(t\right)%
\mathrm{d}t=\frac{\delta_{\ell,m}}{2(2\ell+\nu+1)},
\end{align*}
where $\nu+\ell+m>-1$.

\subsection{Meijer-G representations of classical functions}\label{sec:meijergforms}
\cref{thm:meijergrepresentationtheorem} presents a Meijer-G function representation for all the one-dimensional weighted classical orthogonal polynomials as well as Bessel functions of the first and second kind we consider in this paper. Note that while the Jacobi polynomial case includes the Legendre, Chebyshev, and Gegenbauer polynomials as special cases (up to a constant rescaling) we nevertheless include these and other interesting special case variants for completion. To our knowledge the most complete collection of Meijer-G forms of classical functions is contained in Volume 3 of the formula collection due to Prudnikov, Byrchkov and Marichev, see \cite[8.4]{prudnikovVol3}.

\begin{remark}
In our choices for the Meijer-G representations in \cref{thm:meijergrepresentationtheorem}, we intentionally avoid certain choices commonly made in the literature: in particular the case separation of even and odd degree. We thus need a function of $n \in \mathbb{N}_0$ which distinguishes between even and odd $n$. In code, such functions usually have names such as $\mathrm{isodd}()$ or $\mathrm{iseven}()$. A convenient way to define such a function in a mathematical context is
\begin{align*}
\mathrm{isodd}(n) \coloneqq  n - 2\left\lfloor \frac{n}{2} \right\rfloor &= \begin{cases} 0 \quad  \text{if } \frac{n}{2} \in \mathbb{N}_0 \\   
1 \quad \text{if } \frac{n}{2} \not\in \mathbb{N}_0 \end{cases} \quad \text{for } n \in \mathbb{N}_0.
\end{align*}
Thus in what follows we always have $\left(n - 2\left\lfloor \frac{n}{2} \right\rfloor\right) \in \{0,1\}$.
\end{remark}

\begin{theorem}\label{thm:meijergrepresentationtheorem}
The functions in column 2 of \cref{tab:meijergreptable} are equal to the Meijer-G function representations in column 3.
\begin{proof}
We provide either a proof or a reference for each row of \cref{tab:meijergreptable}. Some formulas do not appear in any paper, book, or reputable formula collection known to the authors. These formulas are proved in the lemmas cited in \cref{sec:app:proofs}.
\begin{enumerate}[start=1,label={(\bfseries Row \arabic*):}]
\item While likely to be known to some mathematicians, we are not aware of explicit references for this general weighted Jacobi polynomial case in Meijer-G form in the literature and thus include a proof in \cref{sec:app:proofs}, see \cref{lem:equaljacobimeijergseries}.
\item Gegenbauer special case of Row 1.
\item First kind Chebyshev special case of Row 1.
\item Second kind Chebyshev special case of Row 1.
\item This simpler radially shifted variant of the weighted Jacobi polynomials is known in the literature, c.f.~\cite[8.4.36.1]{prudnikovVol3}, \cite[05.06.26.0011.01]{WolframFunctions2022} or \cite{dyda2017fractional}.
\item While we are not aware of this explicit general form being present in the literature, the formula in \cite[8.4.34.3]{prudnikovVol3} (see also \cite[05.01.26.0010.01]{WolframFunctions2022}) is a closely related special case from which one can derive the presented form using properties of the Hermite polynomials.
\item Known via \cite[8.4.37.1]{prudnikovVol3} (cf.~\cite[05.08.26.0003.01]{WolframFunctions2022}).
\item Known via \cite[8.4.19.1]{prudnikovVol3} (cf.~\cite[03.01.26.0007.01]{WolframFunctions2022}).
\item Known via \cite[8.4.19.5]{prudnikovVol3} (cf.~\cite[03.01.26.0011.01]{WolframFunctions2022}).
\item Known via \cite[8.4.19.5]{prudnikovVol3} (cf.~\cite[03.01.26.0013.01]{WolframFunctions2022}).
\item Known via \cite[8.4.19.15]{prudnikovVol3} (cf.~\cite[03.01.26.0021.01]{WolframFunctions2022}), only valid if $-\mu-\nu-1 \notin \mathbb{N}$.
\item Known via \cite[8.4.20.1]{prudnikovVol3} (cf.~\cite[03.03.26.0003.01]{WolframFunctions2022}).
\item Known via \cite[8.4.20.3]{prudnikovVol3} (cf.~\cite[03.03.26.0007.01]{WolframFunctions2022}).
\item Known via \cite[8.4.20.3]{prudnikovVol3} (cf.~\cite[03.03.26.0008.01]{WolframFunctions2022}).
\end{enumerate}
\end{proof}
\end{theorem}
\begin{remark}
A factor of $(x^2)^\sigma$, with $\sigma \in \mathbb{R}$, may be multiplied with all of the functions listed in \cref{thm:meijergrepresentationtheorem} and included in the Meijer-G form via the multiplicative shift formula in \cref{eq:meijerGsym2}.
\end{remark}

\subsection{The fractional Laplacian and the Riesz potential}\label{sec:fractionallapandriesz}
Few mathematical operators have been studied more than the classical Laplace operator, $\Delta$, and entire families of functions are commonly defined via its action.
\begin{definition}[Solid harmonic polynomials]
For given $d \in \mathbb{N}$ and $\ell\geq0$, a solid harmonic polynomial of degree $\ell$ is a homogeneous polynomial $V_\ell(\mathbf{x})$ satisfying $$\Delta V_\ell(\mathbf{x}) = 0, \quad \mathbf{x} \in \mathbb{R}^d.$$
\end{definition}
In recent years, the so-called fractional generalization of the Laplacian has received much attention both in applied and pure research. While it is possible to introduce the fractional Laplacian in at least ten equivalent ways, as seen in \cite{Kwasnicki2017}, we present three definitions of the fractional Laplacian which are commonly useful for applications and give insight into its behavior: a characterization in terms of its Fourier transform, one in terms of singular integrals, and one in terms of the inverse of the Riesz potential.

Suppose that $u, \Delta u \in L^p(\mathbb{R}^d)$ for some $p \in [1,\infty)$. Then it is a classical result that the Laplace operator $\Delta$ has the Fourier multiplier representation
\begin{align}
\mathcal{F}[\Delta u](\xi) = -|\xi|^2 \mathcal{F}[u](\xi),
\end{align}
where $\mathcal{F}[u](\xi)$ denotes the Fourier transform of $u(x)$ as a function of $\xi$, i.e. 
\begin{align}
\mathcal{F}[u](\xi) \coloneqq \int_{\mathbb{R}^d} u(x)\E^{-\im x\cdot\xi} \mathrm{d}x.
\end{align}
The spectral theorem \cite[Thm.~1.1]{Kwasnicki2017} then suggests the following Fourier multiplier definition of the fractional Laplacian:
\begin{definition}[Fractional Laplacian (Fourier)]\label{def:fourierfractionallaplace} Let $s>0$ and $f\in\mathcal{S}(\mathbb{R}^d)$. We define the fractional Laplacian by
\begin{align}
(-\Delta)^s f\coloneqq\mathcal{F}^{-1}\!\left[|\xi|^{2s}\mathcal{F}[f](\xi)\right].
\end{align}
\end{definition}
Note that some authors prefer to write $(-\Delta)^{\frac{\alpha}{2}}$ instead of $(-\Delta)^{s}$ since the former (with $0<\alpha<2$) agrees with the order definition of pseudo-differential operators. Furthermore, others prefer to refer to $(-\Delta)^{s}$ as the negative fractional Laplacian and reserve the term fractional Laplacian for $-(-\Delta)^{s}$, cf.~\cite{Kwasnicki2017}. We will write $(-\Delta)^s$ and call this object the fractional Laplacian throughout this paper. On any common domain for which the multipliers and compositions are defined, the Fourier definition readily implies the semigroup property
\begin{align*}
(-\Delta)^{a}(-\Delta)^{b} = (-\Delta)^{a+b}.
\end{align*}
One also finds that the fractional Laplacian defined in this way is a non-local operator. We now briefly discuss the Riesz potential before returning to the above definition of the fractional Laplacian from a different perspective.
\begin{definition}[Riesz potential]
Let $s \in \left(0,\frac{d}{2}\right)$, and let $f \in L^p(\mathbb{R}^d)$ with $1 \leq p < \frac{d}{2s}$. The Riesz potential, defined for almost every $x\in\mathbb{R}^d$, is
\begin{align*}
(-\Delta)^{-s} f(x) \coloneqq \frac{\Gamma\left(\frac{d}{2}-s\right)}{4^{s} \pi^{\frac{d}{2}}\Gamma(s)} \int_{\mathbb{R}^d} \frac{f(y)}{|x-y|^{d-2s}} \mathrm{d}y.
\end{align*}
\end{definition}
The Riesz potential has the Fourier multiplier representation 
\begin{align*}
\mathcal{F}[(-\Delta)^{-s} f](\xi) = |\xi|^{-2s} \mathcal{F}[f](\xi).
\end{align*}
Within the appropriate parameter ranges and function spaces where both operators are well-defined, the Riesz potential is in fact the inverse of the fractional Laplacian, cf.~\cite{Kwasnicki2017}. This serves as our second definition:
\begin{definition}[Fractional Laplacian (Riesz)]\label{def:rieszlaplacian}
Let $s \in \left(0,\frac{d}{2}\right)$. If $f$ and $g$ belong to a function space on which the Riesz potential is defined and
\begin{align}
f=(-\Delta)^{-s}g,
\end{align}
then we define $(-\Delta)^s f\coloneqq g$. Thus the fractional Laplacian is the inverse of the Riesz potential on its operator domain.
\end{definition}
It is commonplace to build this connection to the fractional Laplacian into the notation for the Riesz potential. In the potential theory literature one would more commonly find the notation $I_{2s}$ for the Riesz potential defined as above. We refer to \cite{Kwasnicki2017} for a much more detailed discussion of the relationship between the Riesz potential and the fractional Laplacian.

The singular-integral definition of the fractional Laplacian is motivated by the observation that the fractional Laplacian is a non-local operator and the inverse of the Riesz potential. We first give its familiar form for $0<s<1$ and then, following \cite[Section~1.1]{dyda2017fractional}, its higher-order extension.
\begin{definition}[Fractional Laplacian (singular integral)] 
Let $s \in (0,1)$ and $f \in \mathcal{S}(\mathbb{R}^d)$ where $\mathcal{S}(\mathbb{R}^d)$ denotes the space of Schwartz functions. We define the fractional Laplacian $(-\Delta)^s : \mathcal{S}(\mathbb{R}^d) \to L^2(\mathbb{R}^d)$ via the singular integral
\begin{align}
(-\Delta)^sf(x) \coloneqq \frac{4^s\Gamma(\frac{d}{2}+s)}{\pi^{\frac{d}{2}}|\Gamma(-s)|} \dashint_{\mathbb{R}^d}\frac{f(x)-f(y)}{|x-y|^{d+2s}} \mathrm{d}y.
\end{align}
Here $\dashint_{\mathbb{R}^d} \cdot$ denotes the Cauchy principal value integral \cite[Ch.~2.4]{King2009a}.
\end{definition}
\begin{definition}[Fractional Laplacian (higher-order singular integral)]\label{def:higherorderfractionallaplace}
Let $s>0$, let $f\in\mathcal{S}(\mathbb{R}^d)$, and choose an even integer $k>2s$. Define the centered difference
\begin{align*}
\Delta_y^k f(x)\coloneqq\sum_{j=0}^k(-1)^j\binom{k}{j}f\left(x+\left(\frac{k}{2}-j\right)y\right)
\end{align*}
and
\begin{align*}
\gamma_d(\alpha)&\coloneqq\frac{2^\alpha\pi^{d/2}\Gamma(\alpha/2)}{\Gamma((d-\alpha)/2)},\\
\chi_{d,k}(2s)&\coloneqq-\gamma_d(-2s)\sum_{j=0}^k(-1)^j\binom{k}{j}\left|\frac{k}{2}-j\right|^{2s}.
\end{align*}
Then
\begin{align}
(-\Delta)^s f(x)\coloneqq\frac{1}{\chi_{d,k}(2s)}\lim_{\varepsilon\to0^+}
\int_{\mathbb{R}^d\setminus B(0,\varepsilon)}
\frac{-\Delta_y^k f(x)}{|y|^{d+2s}}\,\mathrm{d}y.
\end{align}
At positive integer values of $s$, the poles of $\gamma_d(-2s)$ are cancelled by the corresponding zeros of the finite sum in $\chi_{d,k}(2s)$, so the normalization is understood by analytic continuation. This definition is independent of the admissible choice of $k$ and has Fourier multiplier $|\xi|^{2s}$ \cite[Sec.~1.1]{dyda2017fractional}.
\end{definition}
\begin{remark}
For $0<s<1$, one may take $k=2$, and \cref{def:higherorderfractionallaplace} reduces to the principal-value definition above. More generally, the cancellation in the centered difference gives $\Delta_y^k f(x)=O(|y|^k)$ as $y\to0$, so local integrability follows from $k>2s$. This is the definition underlying the $s>0$ range in \cref{thm:meijergDKKtheorem1}.
\end{remark}
As mentioned above, the Fourier, Riesz and singular-integral definitions for the fractional Laplacian are equivalent in the appropriate parameter ranges provided the domain is $\mathbb{R}^d$ (the case of bounded domains is much more subtle \cite{garbaczewski2019fractional}). We refer to \cite{Kwasnicki2017,dyda2017fractional} for detailed arguments and higher-order extensions.

\subsection{A brief history of explicit form fractional Laplacians}\label{sec:briefhistory}
While we do not make any claims to completeness, our aim in this section is to give an account of known explicit form results for fractional Laplacians and Riesz potentials, culminating in recent results which motivate the present paper.

M.~Riesz obtained expressions for the harmonic measure and the Green function for balls and complements of the ball in \cite{riesz1938integrales,riesz1938rectification} using Kelvin transforms, cf.~\cite[Appendix]{landkof1972foundations} and \cite{blumenthal1961distribution}. In \cite{hmissi1994fonctions}, cf.~\cite{bogdan1999representation}, it was shown that if $\Delta h = 0$, in the unit ball, then $(-\Delta)^{s} (1-|\mathbf{x}|^2)_+^{s-1} h(\mathbf{x}) = 0$ in the unit ball. Explicit fractional Laplacians and Riesz potentials of profile functions such as $(1-x^2)^s_+$ and $(x^2-1)^s_+$ were discussed in \cite{biler2011barenblatt,biler2015nonlocal,huang2014explicit} in the context of non-local porous medium equations and used in \cite{Jcarrilloradial,gutleb2020computing} to derive and compute explicit steady state solutions to certain power law equilibrium measure problems featuring Riesz potentials. A number of explicit form fractional Laplacian results of functions such as $f_1(\mathbf{x}) = |\mathbf{x}|^{-a}$ with $a \in (0,d)$ and $f_2(\mathbf{x}) = \E^{-|\mathbf{x}|^2}$ were given in \cite{rubin1996fractional}. 

In 2017, Dyda et al.~published a paper describing an efficient numerical scheme for obtaining two-sided bounds for the eigenvalues of the fractional Laplacian in the arbitrary dimensional unit ball \cite{dyda2017eigenvalues} and used it to partially resolve a conjecture of Kulczycki asserting that the second smallest eigenvalue corresponds to an antisymmetric eigenfunction \cite{banuelos2004cauchy, fall2021morse}. They simultaneously published a companion paper \cite{dyda2017fractional} which laid the foundation of their numerical method by describing a general procedure to obtain fractional Laplacians and Riesz potentials of Meijer-G functions, which we reproduce in Theorems \labelcref{thm:meijergDKKtheorem1} and \labelcref{thm:meijergDKKtheorem2} below. In their paper, Dyda et al.~substantially extended the tools for deriving explicit expressions for fractional Laplacians and Riesz potentials of classical functions and applied their results to the example of weighted radially symmetric Jacobi polynomials on ball domains, see also \cite{gutleb2022computation}. Sections \labelcref{sec:fractionalexplicitforms1D} and \labelcref{sec:fractionalexplicitformshighD} of the present paper leverage results from  \cite{dyda2017fractional} to provide explicit form fractional Laplacians and Riesz potentials for most weighted classical orthogonal polynomials as well as Bessel functions. 

The following two theorems reproduce the essential Meijer-G function results of \cite{dyda2017fractional}. While the theorem statements and conditions may appear involved at first glance, their core statement is simple: Given mild assumptions, then the fractional Laplacian or Riesz potential maps a certain family of Meijer-G functions to another family of Meijer-G functions.  
\begin{theorem}[DKK--Meijer-G Theorem]\label{thm:meijergDKKtheorem1}
Let $f(\mathbf{x})$ be a function defined by the product
\begin{align*}
f(\mathbf{x}) \coloneqq V_\ell(\mathbf{x}) G^{m,n}_{p,q} \left(|\mathbf{x}|^2\left|
\begin{array}{c}
 \mathbf{a} \\
 \mathbf{b} \\
\end{array}\right.
\right),
\end{align*}
where $V_\ell(\mathbf{x})$ is a solid harmonic polynomial of degree $\ell \geq 0$ and $\mathbf{x} \in \mathbb{R}^d$. Define 
\begin{align*}
\nu_G\coloneqq\sum_{j=1}^p\operatorname{Re}(a_j)-\sum_{j=1}^q\operatorname{Re}(b_j),
\end{align*}
and adopt the conventions $\max_{1\leq j\leq n}\operatorname{Re}(a_j)=-\infty$ if $n=0$ and $\min_{1\leq j\leq m}\operatorname{Re}(b_j)=+\infty$ if $m=0$. Assume the pole-separation condition
\begin{align}
1-\max_{1\leq j\leq n}\operatorname{Re}(a_j)
> -\min_{1\leq j\leq m}\operatorname{Re}(b_j),
\label{eq:conditionS}
\end{align}
referred to as Condition S in \cite{dyda2017fractional}. Further, define
\begin{equation*}
\bar{\lambda} =
\begin{cases}
1-\max\limits_{1 \leq j \leq n}\operatorname{Re}(a_j),
& \substack{\text{if }p+q<2m+2n,\ \text{or}\\
             p+q=2m+2n\ \text{and }p\geq q,} \\[1ex]
\displaystyle
\min\left(1-\max\limits_{1 \leq j \leq n}\operatorname{Re}(a_j),
\frac{1}{2}+\frac{\nu_G-1}{q-p}\right),
& \substack{\text{if }p+q=2m+2n\\
             \text{and }p<q.}
\end{cases}
\end{equation*}
\begin{equation*}
\underset{\bar{}}{\lambda} =
\begin{cases}
-\min\limits_{1 \leq j \leq m}\operatorname{Re}(b_j),
& \substack{\text{if }p+q<2m+2n,\ \text{or}\\
             p+q=2m+2n\ \text{and }p\leq q,} \\[1ex]
\displaystyle
\max\left(-\min\limits_{1 \leq j \leq m}\operatorname{Re}(b_j),
\frac{1}{2}-\frac{\nu_G-1}{p-q}\right),
& \substack{\text{if }p+q=2m+2n\\
             \text{and }p>q.}
\end{cases}
\end{equation*}

\bigskip
\noindent (i) Suppose that the following inequalities hold:
\begin{subequations}
\begin{align}
s&>0,\\
p+q &< 2m+2n,\\
1-\max_{1 \leq j \leq n} \mathrm{Re}(a_j)&>\frac{\ell-2s}{2},\\
\min_{1 \leq j \leq m} \mathrm{Re}(b_j)&>-\frac{d+\ell}{2}.
\end{align}
\end{subequations}
Then
\begin{align}
(-\Delta)^s f(\mathbf{x}) = 4^s V_\ell(\mathbf{x}) G^{m+1,n+1}_{p+2,q+2} \left(|\mathbf{x}|^2\left|
\begin{array}{c}
 1-s-\frac{d+2\ell}{2}, \mathbf{a}-s, -s \\
 0, \mathbf{b}-s, 1-\frac{d+2\ell}{2} \\
\end{array}\right.
\right),
\end{align}
for all $\mathbf{x} \in \mathbb{R}^d$ such that $\mathbf{x} \neq \mathbf{0}$.

\bigskip
\noindent (ii) The same statement as in (i) holds if instead of the condition $p+q < 2m+2n$ we have $p+q = 2m+2n$ along with $\bar{\lambda}>\frac{-2s+\ell}{2}$ and $\underset{\bar{}}{\lambda}<\frac{d+\ell}{2}$. When $p=q$, we additionally require $\nu_G>0$ if $|\mathbf{x}|\neq 1$, and $\nu_G>1+2s$ if $|\mathbf{x}|=1$.
The results also hold at $\mathbf{x}=0$ if both $f(\mathbf{x})$ and $(-\Delta)^s f(\mathbf{x})$ are continuous at $0$.
\end{theorem}
\begin{proof}
See the proof and discussion of Theorem 2 in \cite{dyda2017fractional}.
\end{proof}
We also require the following analogous result for Riesz potentials:
\begin{theorem}\label{thm:meijergDKKtheorem2}
Let $f(\mathbf{x}), V_\ell(\mathbf{x}), \nu_G, \bar{\lambda}$ and $\underset{\bar{}}{\lambda}$ be defined as in \cref{thm:meijergDKKtheorem1}, and assume Condition S in \eqref{eq:conditionS}.

\noindent (i) Suppose that the following inequalities hold:
\begin{subequations}
\begin{align}
0 < &s < d/2,\\
p+q &< 2m+2n,\\
1-\max_{1 \leq j \leq n} \mathrm{Re}(a_j)&>\frac{\ell+2s}{2},\\
\min_{1 \leq j \leq m} \mathrm{Re}(b_j)&>-\frac{d+\ell}{2}.
\end{align}
\end{subequations}
Then
\begin{align}
(-\Delta)^{-s} f(\mathbf{x}) = 4^{-s} V_\ell(\mathbf{x}) G^{m+1,n+1}_{p+2,q+2} \left(|\mathbf{x}|^2\left|
\begin{array}{c}
 1+s-\frac{d+2\ell}{2}, \mathbf{a}+s, s \\
 0, \mathbf{b}+s, 1-\frac{d+2\ell}{2} \\
\end{array}\right.
\right),
\end{align}
for all $x \in \mathbb{R}^d$ such that $\mathbf{x} \neq \mathbf{0}$.

\bigskip
\noindent (ii) The same statement as in (i) holds if instead of the condition $p+q < 2m+2n$ we have $p+q = 2m+2n$ along with $\bar{\lambda}>\frac{2s+\ell}{2}$ and $\underset{\bar{}}{\lambda}<\frac{d+\ell}{2}$. In the case where $p=q$ we further require $\nu_G>0$ and either $|\mathbf{x}|\neq 1$ or $\nu_G>1-2s$.
\end{theorem}
\begin{proof}
See the proof and discussion of Theorem 1 in \cite{dyda2017fractional}.
\end{proof}

These two results are ultimately a consequence of the fractional Laplacian and Riesz potential's favorable behavior with respect to Mellin integrals as described in \cite{dyda2017fractional} -- in fact the interaction of Mellin integrals and fractional Laplacians is yet another proposed equivalent way to define the fractional Laplacian as recently discussed in \cite{pagnini2023mellin}.

Independent of the Meijer-G approach championed by Dyda et al., explicit forms and interrelations for the fractional Laplacian of weighted Chebyshev polynomials of first and second kind were derived by Papadopoulos and Olver in \cite{papadopoulos2022sparse} for the purpose of a sparse spectral method sum space approach to solving equations of the form 
\begin{align}
\left(\lambda \mathcal{I} + \mu \mathcal{H} + \eta \frac{\mathrm{d}}{\mathrm{d}x}+(-\Delta)^{\frac{1}{2}}\right)u = f,
\end{align}
where $\mathcal{I}$ denotes the identity operator, $\mathcal{H}$ denotes the Hilbert transform (see \cite[1.14(v)]{nist_2018}) and $\lambda, \mu, \eta \in \mathbb{R}$. Finally, Li et al.~recently derived explicit form solutions of the fractional Laplacian for certain weighted Hermite polynomials in \cite{Li2021} and fractional Laplacians on ellipsoids were discussed by Abatangelo et al.~in \cite{abatangelo_fractional_2021}.

\begin{remark}
Although not utilized in this paper, we note that the following older, but also quite general, relationship due to Bateman \cite{bateman1909solution} may be used to derive many explicit forms and recurrence relationships for Riemann--Liouville fractional integrals such as $\frac{1}{\Gamma
(1+\alpha)} \int_{-1}^{x} (x-y)^{\alpha} u(y) \mathrm{d}y$
for some of the functions discussed in this paper, see e.g.~\cite{chen2016generalized,gutleb2020computing} for Jacobi polynomials:
\begin{align*}
{}_2F_1\left(\left. \begin{array}{c} a, b\\ c+\rho \end{array} \right| x \right) = \frac{\Gamma(c+\rho)}{\Gamma(c)\Gamma(\rho)}x^{1-c-\rho}\int_0^x t^{c-1} (x-t)^{\rho-1} {}_2F_1\left(\left. \begin{array}{c} a,b \\ c \end{array} \right| t \right) \mathrm{d}t, \quad |x|<1.
\end{align*}
Riemann-Liouville fractional integrals are closely related to one-dimensional Riesz potentials.
\end{remark}

\newpage
    \clearpage
    \thispagestyle{empty}
    {\newgeometry{left=1cm,right=1cm,bottom=1cm,top=1cm}
    \begin{landscape}
        \centering 
\begin{tabular}{c c c}
 \textbf{Row ID}& \textbf{Classical notation}         &      \textbf{Meijer-G form}    \\ \hline \hline
 \textbf{1} & $ (1-x^2)_+^a P_n^{(a,a)}(x) $          & $\frac{\Gamma (a+n+1)}{n!} x^{n-2 \left\lfloor \frac{n}{2}\right\rfloor } G_{2,2}^{2,0}\left(x^2 \left|
\begin{array}{c}
 a+\left\lfloor \frac{n}{2}\right\rfloor +1,-n+\left\lfloor \frac{n}{2}\right\rfloor +\frac{1}{2} \\
 0,-n+2 \left\lfloor \frac{n}{2}\right\rfloor +\frac{1}{2} \\
\end{array}\right.
\right)$ \\   
\textbf{2} & $ (1-x^2)_+^{\lambda-\frac{1}{2}} C_n^{(\lambda)}(x) $           & $ \frac{\Gamma (\lambda+n+\frac{1}{2}) \left(2\lambda\right)_{n}}{n! \left(\lambda+\frac{1}{2}\right)_{n}} x^{n-2 \left\lfloor \frac{n}{2}\right\rfloor } G_{2,2}^{2,0}\left(x^2 \left|
\begin{array}{c}
 \lambda+\left\lfloor \frac{n}{2}\right\rfloor +\frac{1}{2},-n+\left\lfloor \frac{n}{2}\right\rfloor +\frac{1}{2} \\
 0,-n+2 \left\lfloor \frac{n}{2}\right\rfloor +\frac{1}{2} \\
\end{array}\right.
\right)$         \\  
\textbf{3} & $ (1-x^2)^{-1/2}_+ T_n(x) $          &  $\frac{\Gamma (n+\frac{1}{2})}{n! P^{(-\nicefrac{1}{2},-\nicefrac{1}{2})}_{n}\left(1\right)} x^{n-2 \left\lfloor \frac{n}{2}\right\rfloor } G_{2,2}^{2,0}\left(x^2 \left|
\begin{array}{c}
 \left\lfloor \frac{n}{2}\right\rfloor +\frac{1}{2},-n+\left\lfloor \frac{n}{2}\right\rfloor +\frac{1}{2} \\
 0,-n+2 \left\lfloor \frac{n}{2}\right\rfloor +\frac{1}{2} \\
\end{array}\right.
\right)$       \\  
\textbf{4} & $ (1-x^2)^{1/2}_+ U_n(x) $          &  $\frac{(n+1)\Gamma (n+\frac{3}{2})}{n! P^{(\nicefrac{1}{2},\nicefrac{1}{2})}_{n}(1)} x^{n-2 \left\lfloor \frac{n}{2}\right\rfloor } G_{2,2}^{2,0}\left(x^2 \left|
\begin{array}{c}
 \left\lfloor \frac{n}{2}\right\rfloor +\frac{3}{2},-n+\left\lfloor \frac{n}{2}\right\rfloor +\frac{1}{2} \\
 0,-n+2 \left\lfloor \frac{n}{2}\right\rfloor +\frac{1}{2} \\
\end{array}\right.
\right)$       \\  
\textbf{5} &  $(1-x^2)_+^a(x^2)^b P_n^{(a,b)}(2x^2-1)$           & $\frac{\Gamma (a+n+1) }{n!} G_{2,2}^{2,0}\left(x^2 \left\lvert
\begin{array}{cc}
 a+b+n+1, & -n \\
 0, & b \\
\end{array}\right.
\right)$        \\  
\textbf{6} & $ \E^{-x^2} H_n(x) $           & $2^n x^{n-2 \left\lfloor \frac{n}{2}\right\rfloor } G_{1,2}^{2,0}\left(x^2 \left|
\begin{array}{c}
 -n+\left\lfloor \frac{n}{2}\right\rfloor +\frac{1}{2} \\
 0,-n+2 \left\lfloor \frac{n}{2}\right\rfloor +\frac{1}{2} \\
\end{array}
\right.\right)$        \\  
\textbf{7} & $ \E^{-x^2} L^{\alpha}_n(x^2) $           & $\frac{1}{n!} G_{1,2}^{1,1}\left(x^2\left|
\begin{array}{c}
 -n-\alpha  \\
 0,  -\alpha  \\
\end{array}
\right.\right)$        \\   
\textbf{8} &   $J_{\nu }(2 |x|)$        &  $G_{0,2}^{1,0}\left(x^2 \left|
\begin{array}{c}
 \frac{\nu }{2},-\frac{\nu }{2} \\
\end{array}\right.
\right)$      \\ 
\textbf{9} &   $\cos(a+|x|) J_{\nu }(|x|)$        &  $\frac{1}{\sqrt{2}} G_{3,5}^{2,2}\left(x^2\left|
\begin{array}{c}
 \frac{1}{4},\frac{3}{4},\frac{a}{\pi }+\frac{\nu +1}{2} \\
 \frac{\nu }{2},\frac{\nu +1}{2},-\frac{\nu }{2},\frac{1-\nu }{2},\frac{a}{\pi }+\frac{\nu +1}{2} \\
\end{array}\right.
\right)$      \\ 
\textbf{10} &   $\sin(a+|x|) J_{\nu }(|x|)$        &  $\frac{1}{\sqrt{2}}G_{3,5}^{2,2}\left(x^2\left|
\begin{array}{c}
 \frac{1}{4},\frac{3}{4},\frac{a}{\pi }+\frac{\nu }{2} \\
 \frac{\nu }{2},\frac{\nu +1}{2},-\frac{\nu }{2},\frac{1-\nu }{2},\frac{a}{\pi }+\frac{\nu }{2} \\
\end{array}\right.
\right)$      \\ 
\textbf{11} &   $J_{\mu }(|x|) J_{\nu }(|x|)$        &  $\frac{1}{\sqrt{\pi }}G_{2,4}^{1,2}\left(x^2 \left|
\begin{array}{c}
 0,\frac{1}{2} \\
 \frac{\mu +\nu }{2},-\frac{\mu +\nu}{2},\frac{\mu -\nu }{2},\frac{\nu -\mu }{2} \\
\end{array}\right.
\right)$      \\ 
\textbf{12} &   $Y_{\nu }(2 |x|)$        &$G_{1,3}^{2,0}\left(x^2\left|
\begin{array}{c}
 -\frac{\nu +1}{2} \\
 \frac{\nu }{2},-\frac{\nu }{2},-\frac{\nu +1}{2}  \\
\end{array}\right.\right)$\\
\textbf{13} &   $\cos(|x|) Y_{\nu }(|x|)$        &$\frac{1}{\sqrt{2}}G_{3,5}^{2,2}\left(x^2\left|
\begin{array}{c}
 \frac{1}{4},\frac{3}{4},-\frac{\nu +1}{2} \\
 -\frac{\nu }{2},\frac{\nu }{2},-\frac{\nu +1}{2} ,\frac{1-\nu }{2},\frac{\nu +1}{2} \\
\end{array}\right.
\right)$\\
\textbf{14} &   $\sin(|x|) Y_{\nu }(|x|)$        &$\frac{1}{\sqrt{2}}G_{3,5}^{2,2}\left(x^2 \left|
\begin{array}{c}
 \frac{1}{4},\frac{3}{4},-\frac{\nu }{2} \\
 \frac{\nu +1}{2},\frac{1-\nu }{2},-\frac{\nu }{2},-\frac{\nu }{2},\frac{\nu }{2} \\
\end{array} \right.
\right)$
\end{tabular}
\captionof{table}{All formulae in this table are valid for all real values of $x$ where the functions in column 2 are well-defined. Row 11 only holds if $-\mu-\nu-1 \notin \mathbb{N}$.}
\label{tab:meijergreptable}
\end{landscape}}
    \clearpage
    \restoregeometry
    
\section{Fractional Laplacian and Riesz potential formulae in 1D}\label{sec:fractionalexplicitforms1D}
\sectionmark{Formulae in 1D}
We now present explicit forms for the fractional Laplacian and Riesz potential of the functions in \cref{tab:meijergreptable}. Note that, while we do not explicitly provide them, scaled versions are obtained for most formulae in this paper by setting $x=\lambda y$ with $\lambda \in \mathbb{R}$ and using the scaling property of the fractional Laplacian. Shifted versions are likewise obtained by means of variable transforms.

\begin{theorem}\label{thm:fractionallaplacianformstheorem1D}
Given a function $f(x)$ in column 2 of \cref{tab:fractionallaplacetable}, then the explicit expression of $(-\Delta)^s f(x)$ is the expression in column 3, subject to the following sufficient conditions inherited from \cref{thm:meijergDKKtheorem1,thm:meijergDKKtheorem2}:
\begin{enumerate}[label=(\roman*)]
\item For Rows 1--7, either $s>0$, in which case $(-\Delta)^s$ denotes the fractional Laplacian, or $-1/2<s<0$, in which case it denotes the Riesz potential. For Rows 8--10 and 12--14 we require $s>0$. For Row 11 we require either $s>0$ or $-1/4<s<0$.
\item The row-specific parameter conditions are: $a>-1$ in Row 1; $\lambda>-1/2$ and $\lambda\neq0$ in Row 2; $a>-1$ and $b>-1/2$ in Row 5; $\operatorname{Re}(\alpha)>\max(-s-n-1,-n-1)$ in Row 7; $\operatorname{Re}(\nu)>-1$ in Row 8; $\operatorname{Re}(\nu)>-1/2$ in Rows 9--10; $\operatorname{Re}(\mu+\nu)>-1$ in Row 11; $|\operatorname{Re}(\nu)|<1$ in Row 12; $|\operatorname{Re}(\nu)|<1/2$ in Row 13; and $|\operatorname{Re}(\nu)|<3/2$ in Row 14. The restriction $-\mu-\nu-1\notin\mathbb N$ from \cref{tab:meijergreptable} also applies to Row 11.
\item The formulas hold for $x\neq0$. They extend to $x=0$ whenever both the input and the displayed output are continuous there. In Rows 1--5, $|x|=1$ is excluded unless $\sigma>2s$, where $\sigma=a,\lambda-1/2,-1/2,1/2,a$ in Rows 1--5, respectively.
\end{enumerate}
Tables \ref{tab:fractionallaplacetable2} and \ref{tab:fractionallaplacetable2continued} contain the corresponding hypergeometric forms for $(-\Delta)^s f(x)$ where available; values at removable singularities of those representations are understood by limits.
\end{theorem}
\begin{proof}
The proofs for each row in  \cref{tab:fractionallaplacetable} are diligent applications of Theorems \ref{thm:meijergDKKtheorem1} and \ref{thm:meijergDKKtheorem2}. We concisely state the cases of these theorems that the respective proofs fall under.\\
\begin{enumerate}[start=1,label={(\bfseries Row \arabic*):}]
\item Via \cref{tab:meijergreptable}, set $\mathbf{a} = (a+\lfloor\frac{n}{2} \rfloor +1, -n+ \lfloor \frac{n}{2}\rfloor + \frac{1}{2})$ and $\mathbf{b} = (0, -n+2\lfloor \frac{n}{2}\rfloor+\frac{1}{2})$ in \cref{thm:meijergDKKtheorem1,thm:meijergDKKtheorem2}  with $(m,n,p,q) = (2,0,2,2)$ and $\ell = n - 2\lfloor \frac{n}{2} \rfloor$. Then use the permutation symmetry property followed by the cancellation-reduction property in \eqref{eq:meijerGfullcancellation2}.
\item Gegenbauer special case of Row 1.
\item First kind Chebyshev special case of Row 1.
\item Second kind Chebyshev special case of Row 1. For related results for weighted Chebyshev polynomials we also refer to the independent methods described in \cite[Prop.~3.5]{papadopoulos2022sparse}.
\item Via \cref{tab:meijergreptable}, set $\mathbf{a} = (a+b+n+1, -n)$ and $\mathbf{b} = (0,b)$ in \cref{thm:meijergDKKtheorem1,thm:meijergDKKtheorem2} with $(m,n,p,q) = (2,0,2,2)$ and $\ell=0$. Then use the permutation symmetry property followed by the cancellation-reduction property in \eqref{eq:meijerGfullcancellation2}. To our knowledge this formula was first explicitly stated in \cite{dyda2017fractional}.
\item Via \cref{tab:meijergreptable}, set $\mathbf{a} = (-n+\lfloor\frac{n}{2} \rfloor +\frac{1}{2})$ and $\mathbf{b} = (0,-n+2\lfloor\frac{n}{2} \rfloor +\frac{1}{2})$ in \cref{thm:meijergDKKtheorem1,thm:meijergDKKtheorem2} with $(m,n,p,q) = (2,0,1,2)$ and $\ell = n - 2\lfloor \frac{n}{2} \rfloor$. Then use the permutation symmetry property followed by the cancellation-reduction property in \eqref{eq:meijerGfullcancellation2}. A different strategy for a related Hermite polynomial result can be found in \cite{Li2021}.
\item Via \cref{tab:meijergreptable}, set $\mathbf{a} = (-n-\alpha)$ and $\mathbf{b} = (0,-\alpha)$ in \cref{thm:meijergDKKtheorem1,thm:meijergDKKtheorem2} with $(m,n,p,q) = (1,1,1,2)$ and $\ell=0$.
\item Via \cref{tab:meijergreptable}, set $\mathbf{a} = ()$, i.e. containing no entries, and $\mathbf{b} = (\frac{\nu}{2},-\frac{\nu}{2})$  in \cref{thm:meijergDKKtheorem1,thm:meijergDKKtheorem2} with $(m,n,p,q) = (1,0,0,2)$.
\item Via \cref{tab:meijergreptable}, set $\mathbf{a} = (\frac{1}{4}, \frac{3}{4}, \frac{a}{\pi}+\frac{\nu+1}{2})$ and $\mathbf{b} = (\frac{\nu}{2},\frac{\nu+1}{2},-\frac{\nu}{2},\frac{1-\nu}{2}, \frac{a}{\pi}+\frac{\nu+1}{2})$ in \cref{thm:meijergDKKtheorem1,thm:meijergDKKtheorem2} with $(m,n,p,q) = (2,2,3,5)$.
\item Via \cref{tab:meijergreptable}, set  $\mathbf{a} = (\frac{1}{4}, \frac{3}{4}, \frac{a}{\pi}+\frac{\nu}{2})$ and $\mathbf{b} = (\frac{\nu}{2},\frac{\nu+1}{2},-\frac{\nu}{2},\frac{1-\nu}{2}, \frac{a}{\pi}+\frac{\nu}{2})$ in \cref{thm:meijergDKKtheorem1,thm:meijergDKKtheorem2} with $(m,n,p,q) = (2,2,3,5)$.
\item Via \cref{tab:meijergreptable}, set  $\mathbf{a} = (0,\frac{1}{2})$ and $\mathbf{b}=(\frac{\mu+\nu}{2}, -\frac{\mu+\nu}{2}, \frac{\mu-\nu}{2},\frac{\nu-\mu}{2})$ in \cref{thm:meijergDKKtheorem1,thm:meijergDKKtheorem2} with $(m,n,p,q) = (1,2,2,4)$.
\item Via \cref{tab:meijergreptable}, set  $\mathbf{a} = (-\frac{\nu+1}{2})$ and $\mathbf{b} = (\frac{\nu}{2}, -\frac{\nu}{2},-\frac{\nu+1}{2})$ in \cref{thm:meijergDKKtheorem1,thm:meijergDKKtheorem2} with $(m,n,p,q) = (2,0,1,3)$.
\item Via \cref{tab:meijergreptable}, set  $\mathbf{a} = (\frac{1}{4}, \frac{3}{4}, -\frac{\nu+1}{2})$ and $\mathbf{b} = (-\frac{\nu}{2}, \frac{\nu}{2}, -\frac{\nu+1}{2}, \frac{1-\nu}{2}, \frac{\nu+1}{2})$ in \cref{thm:meijergDKKtheorem1,thm:meijergDKKtheorem2} with $(m,n,p,q) = (2,2,3,5)$.
\item Via \cref{tab:meijergreptable}, set  $\mathbf{a} = (\frac{1}{4}, \frac{3}{4}, -\frac{\nu}{2})$ and $\mathbf{b} = (\frac{\nu+1}{2}, \frac{1-\nu}{2}, -\frac{\nu}{2}, -\frac{\nu}{2}, \frac{\nu}{2})$ in \cref{thm:meijergDKKtheorem1,thm:meijergDKKtheorem2} with $(m,n,p,q) = (2,2,3,5)$.
\end{enumerate}
\end{proof}

We describe the general procedure used to derive the results in Tables \ref{tab:fractionallaplacetable2} and \ref{tab:fractionallaplacetable2continued}. Obtaining the hypergeometric forms is arduous but straightforward, as these results essentially rely on reduction formulae such as the ones listed in \cref{sec:app:reduction}. Note that since symbolic software packages such as Wolfram Mathematica are mostly aware of these reduction formulae, they can be used to check these expressions in an automated way (e.g.~by using \texttt{FunctionExpand}[$\cdot$] on the Meijer-G forms). To give some explicit examples: Rows $1$ to $6$ are derived via the hypergeometric reduction $(2,1,3,3)$ in  \cref{tab:meijergreductionappendix}, followed by standard hypergeometric function simplification arguments. Rows $6$ and $7$ are obtained via the $(2,1,2,3)$ and $(1,2,2,3)$ reduction formulae in \cref{tab:meijergreductionappendix}.

All of the special case formulas in \cref{tab:fractionallaplacetabl3continued} are derived from Tables \ref{tab:fractionallaplacetable2} and \ref{tab:fractionallaplacetable2continued}. The row ID number matches the row in Tables \ref{tab:fractionallaplacetable2} and \ref{tab:fractionallaplacetable2continued} (and by extension \cref{tab:fractionallaplacetable}) from which they are derived by applying the indicated special case via limit arguments. Certain special cases appear to not exist for the hypergeometric forms, e.g.~the results for half-Laplacian special case $s=1/2$ . This is an artefact of the hypergeometric representation and not present in the Meijer-G form.

    \clearpage
    \thispagestyle{empty}
{\newgeometry{left=1cm,right=1cm,bottom=1cm,top=1cm}
    \begin{landscape}
        \centering 
\begin{tabular}{c c c}
 \textbf{Row ID}& $\mathbf{f(x)}$         &      $\mathbf{(-\Delta)^s f(x)}$ (Meijer-G form)    \\ \hline \hline  
\textbf{1} & $ (1-x^2)_+^a P_n^{(a,a)}(x) $           &  $\frac{4^s\Gamma (a+n+1)}{n!} x^{n-2 \left\lfloor \frac{n}{2}\right\rfloor } G_{3,3}^{2,1}\left(x^2 \left|
\begin{array}{c}
 \frac{1}{2}-s-(n-2\left\lfloor \frac{n}{2}\right\rfloor ),a+\left\lfloor \frac{n}{2}\right\rfloor +1-s,-n+\left\lfloor \frac{n}{2}\right\rfloor +\frac{1}{2}-s \\
 0,-n+2 \left\lfloor \frac{n}{2}\right\rfloor +\frac{1}{2}-s,\frac{1}{2}-(n-2\left\lfloor \frac{n}{2}\right\rfloor ) \\
\end{array}\right.
\right)$       \\  
\textbf{2} & $ (1-x^2)_+^{\lambda-\frac{1}{2}} C_n^{(\lambda)}(x) $           & $\frac{4^s \Gamma (\lambda+\frac{1}{2}+n) \left(2\lambda\right)_{n}}{n! \left(\lambda+\frac{1}{2}\right)_{n}} x^{n-2 \left\lfloor \frac{n}{2}\right\rfloor } G_{3,3}^{2,1}\left(x^2 \left|
\begin{array}{c}
 \frac{1}{2}-s-(n-2\left\lfloor \frac{n}{2}\right\rfloor ),\lambda+\frac{1}{2}+\left\lfloor \frac{n}{2}\right\rfloor-s,-n+\left\lfloor \frac{n}{2}\right\rfloor +\frac{1}{2}-s \\
 0,-n+2 \left\lfloor \frac{n}{2}\right\rfloor +\frac{1}{2}-s,\frac{1}{2}-(n-2\left\lfloor \frac{n}{2}\right\rfloor ) \\
\end{array}\right.
\right)$        \\  
\textbf{3} & $ (1-x^2)^{-1/2}_+ T_n(x) $             & $\frac{4^s\Gamma (n+\frac{1}{2})}{n! P^{(-\nicefrac{1}{2},-\nicefrac{1}{2})}_{n}\left(1\right)} x^{n-2 \left\lfloor \frac{n}{2}\right\rfloor } G_{3,3}^{2,1}\left(x^2 \left|
\begin{array}{c}
 \frac{1}{2}-s-(n-2\left\lfloor \frac{n}{2}\right\rfloor ),\left\lfloor \frac{n}{2}\right\rfloor +\frac{1}{2}-s,-n+\left\lfloor \frac{n}{2}\right\rfloor +\frac{1}{2}-s \\
 0,-n+2 \left\lfloor \frac{n}{2}\right\rfloor +\frac{1}{2}-s,\frac{1}{2}-(n-2\left\lfloor \frac{n}{2}\right\rfloor ) \\
\end{array}\right.
\right)$  \\  
\textbf{4} & $ (1-x^2)^{1/2}_+ U_n(x) $       & $\frac{4^s(n+1)\Gamma (n+\frac{3}{2})}{n! P^{(\nicefrac{1}{2},\nicefrac{1}{2})}_{n}\left(1\right)} x^{n-2 \left\lfloor \frac{n}{2}\right\rfloor } G_{3,3}^{2,1}\left(x^2 \left|
\begin{array}{c}
 \frac{1}{2}-s-(n-2\left\lfloor \frac{n}{2}\right\rfloor ),\left\lfloor \frac{n}{2}\right\rfloor +\frac{3}{2}-s,-n+\left\lfloor \frac{n}{2}\right\rfloor +\frac{1}{2}-s \\
 0,-n+2 \left\lfloor \frac{n}{2}\right\rfloor +\frac{1}{2}-s,\frac{1}{2}-(n-2\left\lfloor \frac{n}{2}\right\rfloor ) \\
\end{array}\right.
\right)$   \\  
\textbf{5} &  $(1-x^2)_+^a(x^2)^b P_n^{(a,b)}(2x^2-1)$             & $\frac{4^s \Gamma (a+n+1)}{n!} G_{3,3}^{2,1}\left(x^2 \left\lvert
\begin{array}{c}
 \frac{1}{2}-s,a+b+n-s+1,-n-s \\
 0,b-s,\frac{1}{2} \\
\end{array}\right.
\right)$        \\  
\textbf{6} & $ \E^{-x^2} H_n(x) $           & $2^n4^s x^{n-2 \left\lfloor \frac{n}{2}\right\rfloor } G_{2,3}^{2,1}\left(x^2 \left\lvert
\begin{array}{c}
 -n-s+2 \left\lfloor \frac{n}{2}\right\rfloor +\frac{1}{2},-n-s+\left\lfloor \frac{n}{2}\right\rfloor +\frac{1}{2} \\
 0,-n-s+2 \left\lfloor \frac{n}{2}\right\rfloor +\frac{1}{2},-n+2 \left\lfloor \frac{n}{2}\right\rfloor +\frac{1}{2} \\
\end{array}
\right.\right)$   \\  
\textbf{7} & $ \E^{-x^2} L^{\alpha}_n(x^2) $           & $\frac{4^s}{n!} G_{2,3}^{1,2}\left(x^2 \left\lvert
\begin{array}{c}
 \frac{1}{2}-s,-n-s-\alpha  \\
 0,-\alpha -s,\frac{1}{2} \\
\end{array} \right.
\right)$   \\   
\textbf{8} &      $J_{\nu }(2 |x|)$        &   $4^s G_{2,4}^{2,1}\left(x^2\left\lvert
\begin{array}{c}
 \frac{1}{2}-s,-s \\
 0,\frac{\nu}{2}-s,\frac{1}{2},-\frac{\nu }{2}-s \\
\end{array} \right.
\right)$     \\ 
\textbf{9} &     $\cos(a+|x|) J_{\nu }(|x|)$        &   $2^{2 s-\frac{1}{2}} G_{5,7}^{3,3}\left(x^2 \left\lvert
\begin{array}{c}
 \frac{1}{4}-s,\frac{1}{2}-s,\frac{3}{4}-s,-s,\frac{a}{\pi }+\frac{\nu +1}{2}- s \\
 0,\frac{\nu}{2}-s ,\frac{\nu +1}{2}-s ,\frac{1}{2},\frac{1-\nu}{2}-s,-s-\frac{\nu }{2},\frac{a}{\pi }+\frac{\nu +1}{2}-s  \\
\end{array}\right.
\right)$     \\
\textbf{10}      & $\sin(a+|x|) J_{\nu }(|x|)$     &  $2^{2 s-\frac{1}{2}} G_{5,7}^{3,3}\left(x^2 \left\lvert
\begin{array}{c}
 \frac{1}{4}-s,\frac{1}{2}-s,\frac{3}{4}-s,-s,\frac{a}{\pi }+\frac{\nu}{2}- s \\
 0,\frac{\nu}{2}-s ,\frac{\nu +1}{2}-s ,\frac{1}{2},\frac{1-\nu}{2}-s,-s-\frac{\nu }{2},\frac{a}{\pi }+\frac{\nu}{2}-s  \\
\end{array}\right.
\right)$\\
\textbf{11} &   $J_{\mu }(|x|) J_{\nu }(|x|)$        &  $\frac{4^s}{\sqrt{\pi }} G_{4,6}^{2,3}\left(x^2 \left\lvert
\begin{array}{c}
 \frac{1}{2}-s,-s,\frac{1}{2}-s,-s \\
 0,\frac{\mu +\nu }{2}-s,-\frac{\mu +\nu}{2}-s,\frac{\mu -\nu }{2}-s,\frac{\nu -\mu }{2}-s,\frac{1}{2} \\
\end{array}\right.
\right)$     \\ 
\textbf{12} &   $Y_{\nu }(2 |x|)$        & $4^s G_{3,5}^{3,1}\left(x^2\left|
\begin{array}{c}
 \frac{1}{2}-s,-\frac{\nu +1}{2}-s,-s \\
 0,\frac{\nu }{2}-s,-\frac{\nu }{2}-s,-\frac{\nu +1}{2}-s,\frac{1}{2} \\
\end{array}\right.
\right)$\\
\textbf{13} &   $\cos(|x|) Y_{\nu }(|x|)$        &  $2^{2 s-\frac{1}{2}} G_{5,7}^{3,3}\left(x^2 \left\lvert
\begin{array}{c}
 \frac{1}{2}-s,\frac{1}{4}-s,\frac{3}{4}-s,-\frac{\nu +1}{2}-s,-s \\
 0,-\frac{\nu }{2}-s,\frac{\nu }{2}-s,-\frac{\nu +1}{2}-s,\frac{1-\nu }{2}-s,\frac{\nu +1}{2}-s,\frac{1}{2} \\
\end{array}\right.
\right)$\\
\textbf{14} &   $\sin(|x|) Y_{\nu }(|x|)$      & $2^{2 s-\frac{1}{2}} G_{5,7}^{3,3}\left(x^2 \left\lvert
\begin{array}{c}
 \frac{1}{2}-s,\frac{1}{4}-s,\frac{3}{4}-s,-s-\frac{\nu }{2}, -s \\
 0,\frac{1-\nu}{2}-s,\frac{1+\nu}{2}-s,-\frac{\nu }{2}-s,-\frac{\nu }{2}-s,\frac{\nu}{2}-s,\frac{1}{2} \\
\end{array}\right.
\right)$
\end{tabular}
\captionof{table}{Explicit Meijer-G function forms of the fractional Laplacian and Riesz potential of classical functions, subject to the order, parameter and pointwise conditions in the preceding theorem.}
\label{tab:fractionallaplacetable}
    \end{landscape}}
    \clearpage

\newpage
    \clearpage
    \thispagestyle{empty}
    {\newgeometry{left=1cm,right=1cm,bottom=1cm,top=1cm}
    \begin{landscape}
        \centering 
\begin{tabular}{c c}
 \textbf{Row ID}&      $\mathbf{(-\Delta)^s f(x)}$ (Hypergeometric form)   \\ \hline \hline
\textbf{1}          &  $4^{s}\frac{\Gamma (a+n+1)}{n!} x^{n-2 \left\lfloor \frac{n}{2}\right\rfloor }  \begin{cases}  \frac{\pi \, _2F_1\left(-a+s-\left\lfloor \frac{n}{2}\right\rfloor ,n+s-\left\lfloor \frac{n}{2}\right\rfloor +\frac{1}{2};n-2 \left\lfloor \frac{n}{2}\right\rfloor +\frac{1}{2};x^2\right)}{\sin \left(\pi  \left(2 \left\lfloor \frac{n}{2}\right\rfloor - n- s+\frac{1}{2}\right)\right)\Gamma \left(n-2 \left\lfloor \frac{n}{2}\right\rfloor +\frac{1}{2}\right) \Gamma \left(-n-s+\left\lfloor \frac{n}{2}\right\rfloor +\frac{1}{2}\right) \Gamma \left(a-s+\left\lfloor \frac{n}{2}\right\rfloor +1\right)}  , \quad |x|<1  \\
\frac{-2^{-n-2 s} \sin (\pi  s) \Gamma (n+2 s+1) | x| ^{-2 \left\lfloor \frac{n-1}{2}\right\rfloor -2 s-3} \, _2F_1\left(s+\left\lfloor \frac{n}{2}\right\rfloor +1,\frac{2 \left\lfloor \frac{n-1}{2}\right\rfloor +3}{2} +s;\frac{2 n+3}{2}+a;\frac{1}{x^2}\right)}{\sqrt{\pi } \Gamma \left(\frac{2 n+3}{2} +a\right)}, \quad |x|>1  \end{cases}$       \\  
\textbf{2}  & $\frac{4^s \Gamma (\lambda+\frac{1}{2}+n) \left(2\lambda\right)_{n}}{n! \left(\lambda+\frac{1}{2}\right)_{n}} x^{n-2 \left\lfloor \frac{n}{2}\right\rfloor } \begin{cases}  \frac{\pi \, _2F_1\left(-\lambda+\frac{1}{2}+s-\left\lfloor \frac{n}{2}\right\rfloor ,n+s-\left\lfloor \frac{n}{2}\right\rfloor +\frac{1}{2};n-2 \left\lfloor \frac{n}{2}\right\rfloor +\frac{1}{2};x^2\right)}{\sin \left(\pi  \left(2 \left\lfloor \frac{n}{2}\right\rfloor - n- s+\frac{1}{2}\right)\right)\Gamma \left(n-2 \left\lfloor \frac{n}{2}\right\rfloor +\frac{1}{2}\right) \Gamma \left(-n-s+\left\lfloor \frac{n}{2}\right\rfloor +\frac{1}{2}\right) \Gamma \left(\lambda-s+\left\lfloor \frac{n}{2}\right\rfloor +\frac{1}{2}\right)} , \quad |x|<1  \\ \frac{-2^{-n-2 s} \sin (\pi  s) \Gamma (n+2 s+1) | x| ^{-2 \left\lfloor \frac{n-1}{2}\right\rfloor -2 s-3} \, _2F_1\left(s+\left\lfloor \frac{n}{2}\right\rfloor +1,\frac{2 \left\lfloor \frac{n-1}{2}\right\rfloor +3}{2}+s;n+1+\lambda;\frac{1}{x^2}\right)}{\sqrt{\pi } \Gamma \left(n+1+\lambda\right)}, \quad |x|>1 \end{cases}$        \\  
\textbf{3}      & $\frac{4^s\Gamma (n+\frac{1}{2})}{n! P^{(-\nicefrac{1}{2},-\nicefrac{1}{2})}_{n}\left(1\right)} x^{n-2 \left\lfloor \frac{n}{2}\right\rfloor } \begin{cases}  \frac{\pi \, _2F_1\left(\frac{1}{2}+s-\left\lfloor \frac{n}{2}\right\rfloor ,n+s-\left\lfloor \frac{n}{2}\right\rfloor +\frac{1}{2};n-2 \left\lfloor \frac{n}{2}\right\rfloor +\frac{1}{2};x^2\right)}{\sin \left(\pi  \left(2 \left\lfloor \frac{n}{2}\right\rfloor - n- s+\frac{1}{2}\right)\right)\Gamma \left(n-2 \left\lfloor \frac{n}{2}\right\rfloor +\frac{1}{2}\right) \Gamma \left(-n-s+\left\lfloor \frac{n}{2}\right\rfloor +\frac{1}{2}\right) \Gamma \left(-s+\left\lfloor \frac{n}{2}\right\rfloor +\frac{1}{2}\right)} , \quad |x|<1  \\ \frac{-2^{-n-2 s} \sin (\pi  s) \Gamma (n+2 s+1) | x| ^{-2 \left\lfloor \frac{n-1}{2}\right\rfloor -2 s-3} \, _2F_1\left(s+\left\lfloor \frac{n}{2}\right\rfloor +1,\frac{2 \left\lfloor \frac{n-1}{2}\right\rfloor +3}{2}+s;n+1;\frac{1}{x^2}\right)}{\sqrt{\pi } \Gamma \left(n+1 \right)}, \quad |x|>1 \end{cases}$   \\  
\textbf{4}      & $\frac{4^s(n+1)\Gamma (n+\frac{3}{2})}{n! P^{(\nicefrac{1}{2},\nicefrac{1}{2})}_{n}\left(1\right)} x^{n-2 \left\lfloor \frac{n}{2}\right\rfloor } \begin{cases}  \frac{\pi \, _2F_1\left(s-\frac{1}{2}-\left\lfloor \frac{n}{2}\right\rfloor ,n+s-\left\lfloor \frac{n}{2}\right\rfloor +\frac{1}{2};n-2 \left\lfloor \frac{n}{2}\right\rfloor +\frac{1}{2};x^2\right)}{\sin \left(\pi  \left(2 \left\lfloor \frac{n}{2}\right\rfloor - n- s+\frac{1}{2}\right)\right)\Gamma \left(n-2 \left\lfloor \frac{n}{2}\right\rfloor +\frac{1}{2}\right) \Gamma \left(-n-s+\left\lfloor \frac{n}{2}\right\rfloor +\frac{1}{2}\right) \Gamma \left(1-s+\left\lfloor \frac{n}{2}\right\rfloor +\frac{1}{2}\right)}  , \quad |x|<1  \\ \frac{-2^{-n-2 s} \sin (\pi  s) \Gamma (n+2 s+1) | x| ^{-2 \left\lfloor \frac{n-1}{2}\right\rfloor -2 s-3} \, _2F_1\left(s+\left\lfloor \frac{n}{2}\right\rfloor +1,\frac{2 \left\lfloor \frac{n-1}{2}\right\rfloor +3}{2}+s;n+2;\frac{1}{x^2}\right)}{\sqrt{\pi } \Gamma \left(n+2\right)}, \quad |x|>1 \end{cases}$   \\  
\textbf{5}     & $\frac{\Gamma (a+n+1)}{n!} \begin{cases} 4^s\left(\frac{\Gamma \left(s+\frac{1}{2}\right)\Gamma(b-s)\,{}_3F_2\left(s+\frac{1}{2},-a-b-n+s,n+s+1;\frac{1}{2},-b+s+1;x^2\right)}{\sqrt{\pi}\,\Gamma(-n-s)\Gamma(a+b+n-s+1)}\right. \\ \qquad\left.+\frac{(x^2)^{b-s}\Gamma(s-b)\Gamma\left(b+\frac{1}{2}\right)\,{}_3F_2\left(b+\frac{1}{2},-a-n,n+b+1;b-s+1,b-s+\frac{1}{2};x^2\right)}{\Gamma\left(b-s+\frac{1}{2}\right)\Gamma(a+n+1)\Gamma(-b-n)}\right), \quad |x|<1 \\ \frac{-\Gamma \left(b+\frac{1}{2}\right) \sin (\pi  s) \Gamma (2 s+1) | x| ^{-2s-1} \, _3F_2\left(b+\frac{1}{2},s+\frac{1}{2},s+1;\frac{1}{2}-n,a+b+n+\frac{3}{2};\frac{1}{x^2}\right)}{\sqrt{\pi } \Gamma \left(\frac{1}{2}-n\right) \Gamma \left(a+b+n+\frac{3}{2}\right)}, \quad |x|>1 \end{cases}$        \\  
\textbf{6}          & $\tfrac{2^n\pi  4^s x^{n-2 \left\lfloor \frac{n}{2}\right\rfloor } \, _1F_1\left(n+s-\left\lfloor \frac{n}{2}\right\rfloor +\frac{1}{2};n-2 \left\lfloor \frac{n}{2}\right\rfloor +\frac{1}{2};-x^2\right)}{\sin \left(\frac{1}{2} \pi  \left(4 \left\lfloor \frac{n}{2}\right\rfloor -2 n-2 s+1\right)\right)\Gamma \left(n-2 \left\lfloor \frac{n}{2}\right\rfloor +\frac{1}{2}\right) \Gamma \left(-n-s+\left\lfloor \frac{n}{2}\right\rfloor +\frac{1}{2}\right)}$    \\
\textbf{7}       & $\frac{4^s \Gamma \left(s+\frac{1}{2}\right) \Gamma (n+s+\alpha +1) \, _2F_2\left(s+\frac{1}{2},n+s+\alpha +1;\frac{1}{2},s+\alpha +1;-x^2\right)}{\sqrt{\pi } \Gamma (n+1) \Gamma (s+\alpha +1)}$   \\  
\textbf{8}           &     $ \frac{\sin \left(\frac{\pi  \nu }{2}\right) \left(x^2\right)^{\frac{1}{2} (\nu -2 s)} \, _2F_3\left(\frac{\nu }{2}+\frac{1}{2},\frac{\nu }{2}+1;-s+\frac{\nu }{2}+\frac{1}{2},-s+\frac{\nu }{2}+1,\nu +1;-x^2\right)}{\sin \left(\frac{\pi  (\nu -2 s)}{2}\right) \Gamma (-2 s+\nu +1)}-\frac{\sin (\pi  s) \Gamma (2 s+1) \, _2F_3\left(s+\frac{1}{2},s+1;\frac{1}{2},s-\frac{\nu }{2}+1,s+\frac{\nu }{2}+1;-x^2\right)}{ \sin \left(\frac{\pi  (\nu -2 s)}{2} \right) \Gamma \left(s-\frac{\nu }{2}+1\right) \Gamma \left(s+\frac{\nu }{2}+1\right)}$  \\ 
\textbf{9}           &    $ -\frac{2^{-\nu } | x| ^{\nu -2 s}}{\Gamma (-2 s+\nu +1)} \left(\frac{(\nu +1) \sin (a)  \cos \left(\frac{\pi  \nu }{2}\right) | x| \, _4F_5\left(\frac{\nu }{2}+\frac{3}{4},\frac{\nu }{2}+1,\frac{\nu }{2}+\frac{5}{4},\frac{\nu }{2}+\frac{3}{2};\frac{3}{2},-s+\frac{\nu }{2}+1,-s+\frac{\nu }{2}+\frac{3}{2},\nu +1,\nu +\frac{3}{2};-x^2\right)}{(-\nu +2 s-1) \cos \left(\frac{1}{2} \pi  (\nu -2 s+2)\right)}+\frac{\cos (a) \sin \left(\frac{\pi  \nu }{2}\right) \, _4F_5\left(\frac{\nu }{2}+\frac{1}{4},\frac{\nu }{2}+\frac{1}{2},\frac{\nu }{2}+\frac{3}{4},\frac{\nu }{2}+1;\frac{1}{2},-s+\frac{\nu }{2}+\frac{1}{2},-s+\frac{\nu }{2}+1,\nu +\frac{1}{2},\nu +1;-x^2\right)}{\sin \left(\pi  \left(s-\frac{\nu }{2}\right)\right)}\right)$  \\ 
& $+\frac{4^{-s} \sin (\pi  s) \Gamma (4 s+1) \cos \left(a+\frac{\pi  \nu }{2}-\pi  s\right) \, _4F_5\left(s+\frac{1}{4},s+\frac{1}{2},s+\frac{3}{4},s+1;\frac{1}{2},s-\frac{\nu }{2}+\frac{1}{2},s-\frac{\nu }{2}+1,s+\frac{\nu }{2}+\frac{1}{2},s+\frac{\nu }{2}+1;-x^2\right)}{\sin \left(\pi  \left(s-\frac{\nu }{2}\right)\right) \cos \left(\pi  \left(s-\frac{\nu }{2}\right)\right) \Gamma (2 s-\nu +1) \Gamma (2 s+\nu +1)}$\\
\textbf{10}           &  $-\frac{2^{-\nu } | x| ^{\nu -2 s}}{\Gamma (-2 s+\nu +1)} \left(\frac{(\nu +1) \cos (a)  \cos \left(\frac{\pi  \nu }{2}\right) | x| \, _4F_5\left(\frac{\nu }{2}+\frac{3}{4},\frac{\nu }{2}+1,\frac{\nu }{2}+\frac{5}{4},\frac{\nu }{2}+\frac{3}{2};\frac{3}{2},-s+\frac{\nu }{2}+1,-s+\frac{\nu }{2}+\frac{3}{2},\nu +1,\nu +\frac{3}{2};-x^2\right)}{(\nu -2 s+1) \cos \left(\frac{1}{2} \pi  (\nu -2 s+2)\right)}+\frac{\sin (a) \sin \left(\frac{\pi  \nu }{2}\right) \, _4F_5\left(\frac{\nu }{2}+\frac{1}{4},\frac{\nu }{2}+\frac{1}{2},\frac{\nu }{2}+\frac{3}{4},\frac{\nu }{2}+1;\frac{1}{2},-s+\frac{\nu }{2}+\frac{1}{2},-s+\frac{\nu }{2}+1,\nu +\frac{1}{2},\nu +1;-x^2\right)}{\sin \left(\pi  \left(s-\frac{\nu }{2}\right)\right)}\right)$\\
   &  $+\frac{2^{1-2 s} \sin (\pi  s) \Gamma (4 s+1) \sin \left(a+\frac{\pi  \nu }{2}-\pi  s\right) \, _4F_5\left(s+\frac{1}{4},s+\frac{1}{2},s+\frac{3}{4},s+1;\frac{1}{2},s-\frac{\nu }{2}+\frac{1}{2},s-\frac{\nu }{2}+1,s+\frac{\nu }{2}+\frac{1}{2},s+\frac{\nu }{2}+1;-x^2\right)}{\sin (2 \pi  s-\pi  \nu ) \Gamma (2 s-\nu +1) \Gamma (2 s+\nu +1)}$
\end{tabular}
\captionof{table}{Hypergeometric forms for the Meijer-G functions in \cref{tab:fractionallaplacetable} with matching Row IDs, subject to the conditions and joint-limit convention stated above. Formulas with row IDs containing *s are special cases of the corresponding row ID number.}
\label{tab:fractionallaplacetable2}
    \end{landscape}}
    \clearpage

\newpage
    \clearpage
    \thispagestyle{empty}
    {\newgeometry{left=1cm,right=1cm,bottom=1cm,top=1cm}
    \begin{landscape}
        \centering 
\begin{tabular}{c c}
 \textbf{Row ID}&      $\mathbf{(-\Delta)^s f(x)}$ (Hypergeometric form)   \\ \hline \hline
 \textbf{11}           & $ -\frac{2^{-\mu -\nu } \sin \left(\frac{1}{2} \pi  (\mu +\nu )\right) \Gamma (\mu +\nu +1) | x| ^{\mu +\nu -2 s} \, _4F_5\left(\frac{\mu }{2}+\frac{\nu }{2}+\frac{1}{2},\frac{\mu }{2}+\frac{\nu }{2}+\frac{1}{2},\frac{\mu }{2}+\frac{\nu }{2}+1,\frac{\mu }{2}+\frac{\nu }{2}+1;\mu +1,-s+\frac{\mu }{2}+\frac{\nu }{2}+\frac{1}{2},-s+\frac{\mu }{2}+\frac{\nu }{2}+1,\nu +1,\mu +\nu +1;-x^2\right)}{\Gamma (\mu +1) \Gamma (\nu +1) \cos \left(\frac{1}{2} \pi  (\mu +\nu -2 s+1)\right) \Gamma (-2 s+\mu +\nu +1)}$\\
 & $-\frac{4^{-s} \sin (\pi  s) \Gamma (2 s+1)^2 \, _4F_5\left(s+\frac{1}{2},s+\frac{1}{2},s+1,s+1;\frac{1}{2},s-\frac{\mu }{2}-\frac{\nu }{2}+1,s+\frac{\mu }{2}-\frac{\nu }{2}+1,s+\frac{\nu }{2}-\frac{\mu }{2}+1,s+\frac{\mu }{2}+\frac{\nu }{2}+1;-x^2\right)}{\sin \left(\frac{1}{2} \pi  (\mu +\nu -2 s)\right) \Gamma \left(s-\frac{\mu }{2}-\frac{\nu }{2}+1\right) \Gamma \left(s+\frac{\mu }{2}-\frac{\nu }{2}+1\right) \Gamma \left(s+\frac{\nu }{2}-\frac{\mu }{2}+1\right) \Gamma \left(\frac{1}{2} (2 s+\mu +\nu +2)\right)}$\\
 \textbf{12}  &  $\frac{\sqrt{\pi } 2^{\nu +2 s+1} \sin (\pi  s) \Gamma (2 s+1) \, _2F_3\left(s+\frac{1}{2},s+1;\frac{1}{2},s-\frac{\nu }{2}+1,s+\frac{\nu }{2}+1;-x^2\right)}{\sin \left(\frac{1}{2} \pi  (\nu -2 s)\right) \sin \left(\frac{\pi  \nu }{2}+\pi  s\right) \Gamma \left(-s-\frac{\nu }{2}-\frac{1}{2}\right) \Gamma \left(s-\frac{\nu }{2}+1\right) \Gamma (2 s+\nu +2)}-\frac{| x| ^{-\nu -2 s} \, _2F_3\left(\frac{1}{2}-\frac{\nu }{2},1-\frac{\nu }{2};1-\nu ,-s-\frac{\nu }{2}+\frac{1}{2},-s-\frac{\nu }{2}+1;-x^2\right)}{2 \cos \left(\frac{\pi  \nu }{2}\right) \sin \left(\frac{\pi  \nu }{2}+\pi  s\right) \Gamma (-2 s-\nu +1)}$ \\
 & $-\frac{\sin \left(\frac{\pi  \nu }{2}\right) \cos (\pi  \nu ) | x| ^{\nu -2 s} \, _2F_3\left(\frac{\nu }{2}+\frac{1}{2},\frac{\nu }{2}+1;-s+\frac{\nu }{2}+\frac{1}{2},-s+\frac{\nu }{2}+1,\nu +1;-x^2\right)}{\sin (\pi  \nu ) \sin \left(\pi  \left(s-\frac{\nu }{2}\right)\right) \Gamma (-2 s+\nu +1)}$\\
\textbf{13}        & $-\frac{2^{\nu -1} | x| ^{-\nu -2 s} \, _4F_5\left(\frac{1}{4}-\frac{\nu }{2},\frac{1}{2}-\frac{\nu }{2},\frac{3}{4}-\frac{\nu }{2},1-\frac{\nu }{2};\frac{1}{2},\frac{1}{2}-\nu ,1-\nu ,-s-\frac{\nu }{2}+\frac{1}{2},-s-\frac{\nu }{2}+1;-x^2\right)}{\cos \left(\frac{\pi  \nu }{2}\right) \sin \left(\frac{\pi  \nu }{2}+\pi  s\right) \Gamma (-2 s-\nu +1)}-\frac{4^{-s} \sin (\pi  s) \Gamma (4 s+1) \cos \left(\frac{\pi  \nu }{2}+\pi  s\right) \, _4F_5\left(s+\frac{1}{4},s+\frac{1}{2},s+\frac{3}{4},s+1;\frac{1}{2},s-\frac{\nu }{2}+\frac{1}{2},s-\frac{\nu }{2}+1,s+\frac{\nu }{2}+\frac{1}{2},s+\frac{\nu }{2}+1;-x^2\right)}{\sin \left(\frac{1}{2} \pi  (\nu -2 s)\right) \sin \left(\frac{\pi  \nu }{2}+\pi  s\right) \Gamma (2 s-\nu +1) \Gamma (2 s+\nu +1)}$\\
&$-\frac{2^{-\nu } \sin \left(\frac{\pi  \nu }{2}\right) \cos (\pi  \nu ) | x| ^{\nu -2 s} \, _4F_5\left(\frac{\nu }{2}+\frac{1}{4},\frac{\nu }{2}+\frac{1}{2},\frac{\nu }{2}+\frac{3}{4},\frac{\nu }{2}+1;\frac{1}{2},-s+\frac{\nu }{2}+\frac{1}{2},-s+\frac{\nu }{2}+1,\nu +\frac{1}{2},\nu +1;-x^2\right)}{\sin (\pi  \nu ) \sin \left(\pi  \left(s-\frac{\nu }{2}\right)\right) \Gamma (-2 s+\nu +1)}$\\
\textbf{14}       & $\frac{2^{-\nu } (\nu +1) \cos \left(\frac{\pi  \nu }{2}\right) \cos (\pi  \nu ) | x| ^{\nu -2 s+1} \, _4F_5\left(\frac{\nu }{2}+\frac{3}{4},\frac{\nu }{2}+1,\frac{\nu }{2}+\frac{5}{4},\frac{\nu }{2}+\frac{3}{2};\frac{3}{2},-s+\frac{\nu }{2}+1,-s+\frac{\nu }{2}+\frac{3}{2},\nu +1,\nu +\frac{3}{2};-x^2\right)}{\sin (\pi  \nu ) \cos \left(\pi  \left(s-\frac{\nu }{2}\right)\right) \Gamma (-2 s+\nu +2)}-\frac{2^{\nu -1} (\nu -1) | x| ^{-\nu -2 s+1} \, _4F_5\left(\frac{3}{4}-\frac{\nu }{2},1-\frac{\nu }{2},\frac{5}{4}-\frac{\nu }{2},\frac{3}{2}-\frac{\nu }{2};\frac{3}{2},1-\nu ,\frac{3}{2}-\nu ,-s-\frac{\nu }{2}+1,-s-\frac{\nu }{2}+\frac{3}{2};-x^2\right)}{\sin \left(\frac{\pi  \nu }{2}\right) \sin \left(\frac{1}{2} \pi  (\nu +2 s-1)\right) \Gamma (-2 s-\nu +2)}$\\
& $+\frac{4^{-s} \sin (\pi  s) \Gamma (4 s+1) \sin \left(\frac{\pi  \nu }{2}+\pi  s\right) \, _4F_5\left(s+\frac{1}{4},s+\frac{1}{2},s+\frac{3}{4},s+1;\frac{1}{2},s-\frac{\nu }{2}+\frac{1}{2},s-\frac{\nu }{2}+1,s+\frac{\nu }{2}+\frac{1}{2},s+\frac{\nu }{2}+1;-x^2\right)}{\cos \left(\frac{\pi  \nu }{2}+\pi  s\right) \cos \left(\pi  \left(s-\frac{\nu }{2}\right)\right) \Gamma (2 s-\nu +1) \Gamma (2 s+\nu +1)}$
\end{tabular}
\captionof{table}{Continuation of \cref{tab:fractionallaplacetable2}.}
\label{tab:fractionallaplacetable2continued}
    \end{landscape}}
    \clearpage
\newpage

\newpage
    \clearpage
    \thispagestyle{empty}
    {\newgeometry{left=1cm,right=1cm,bottom=1cm,top=1cm}
    \begin{landscape}
        \centering 
\begin{tabular}{c c c c}
 \textbf{Row ID}& $\mathbf{f(x)}$ & \textbf{Special case} &       $\mathbf{(-\Delta)^s f(x)}$   \\ \hline \hline
  \textbf{1*}           & $(1-x^2)_+^s P_n^{(s,s)}(x)$ & $s=a$ & $\begin{cases}  \frac{4^s \Gamma \left(s+\left\lfloor \frac{n}{2}\right\rfloor +1\right) \Gamma \left(n+s-\left\lfloor \frac{n}{2}\right\rfloor +\frac{1}{2}\right)}{\left\lfloor \frac{n}{2}\right\rfloor ! \Gamma \left(n-\left\lfloor \frac{n}{2}\right\rfloor +\frac{1}{2}\right)} P_n^{(s,s)}(x) , \quad |x|<1  \\
-\frac{ \sin (\pi  s) x^{n-2 \left\lfloor \frac{n}{2}\right\rfloor } \Gamma (n+s+1) \Gamma (n+2 s+1) | x| ^{-2 \left\lfloor \frac{n-1}{2}\right\rfloor -2 s-3} \, _2F_1\left(s+\left\lfloor \frac{n-1}{2}\right\rfloor +\frac{3}{2},s+\left\lfloor \frac{n}{2}\right\rfloor +1;n+s+\frac{3}{2};\frac{1}{x^2}\right)}{2^n \sqrt{\pi } n! \Gamma \left(n+s+\frac{3}{2}\right)}, \quad |x|>1  \end{cases}$\\
 \textbf{1**}           & $(1-x^2)_+^a P_n^{(a,a)}(x)$ & $s=\frac{1}{2}$ & $\frac{\Gamma (a+n+1)}{n!} x^{n-2 \left\lfloor \frac{n}{2}\right\rfloor }  \begin{cases}  \frac{2(-1)^{\left\lfloor \frac{n}{2}\right\rfloor } \left\lfloor \frac{n+1}{2}\right\rfloor ! \, _2F_1\left(-a-\left\lfloor \frac{n}{2}\right\rfloor +\frac{1}{2},n-\left\lfloor \frac{n}{2}\right\rfloor +1;n-2 \left\lfloor \frac{n}{2}\right\rfloor +\frac{1}{2};x^2\right)}{\Gamma \left(n-2 \left\lfloor \frac{n}{2}\right\rfloor +\frac{1}{2}\right) \Gamma \left(a+\left\lfloor \frac{n}{2}\right\rfloor +\frac{1}{2}\right)}  , \quad |x|<1  \\
-\frac{2^{-n} \Gamma (n+2) | x| ^{-2 \left\lfloor \frac{n-1}{2}\right\rfloor -4} \, _2F_1\left(\left\lfloor \frac{n}{2}\right\rfloor +\frac{3}{2},\left\lfloor \frac{n-1}{2}\right\rfloor +2;\frac{2 n+3}{2}+a;\frac{1}{x^2}\right)}{\sqrt{\pi } \Gamma \left(\frac{2 n+3}{2}+a\right)}, \quad |x|>1  \end{cases}$\\
 \textbf{1***}           & $(1-x^2)_+^a P_n^{(a,a)}(x)$ & $s=-\frac{1}{2},n\geq 1$ & $\begin{cases}  \frac{(-1)^{\left\lfloor \frac{n}{2}\right\rfloor } \Gamma (a+n+1) \left\lfloor \frac{n-1}{2}\right\rfloor !\,x^{n-2 \left\lfloor \frac{n}{2}\right\rfloor } \, _2F_1\left(-a-\left\lfloor \frac{n}{2}\right\rfloor -\frac{1}{2},n-\left\lfloor \frac{n}{2}\right\rfloor ;n-2 \left\lfloor \frac{n}{2}\right\rfloor +\frac{1}{2};x^2\right)}{2\Gamma(n+1)\Gamma \left(n-2 \left\lfloor \frac{n}{2}\right\rfloor +\frac{1}{2}\right) \Gamma \left(a+\left\lfloor \frac{n}{2}\right\rfloor +\frac{3}{2}\right)}, \quad |x|<1  \\
\frac{2^{-n} \Gamma (n) \Gamma (a+n+1) x^{n-2 \left\lfloor \frac{n}{2}\right\rfloor } | x| ^{-2 \left(\left\lfloor \frac{n-1}{2}\right\rfloor +1\right)} \, _2F_1\left(\left\lfloor \frac{n-1}{2}\right\rfloor +1,\left\lfloor \frac{n}{2}\right\rfloor +\frac{1}{2};a+n+\frac{3}{2};\frac{1}{x^2}\right)}{\sqrt{\pi } n! \Gamma \left(a+n+\frac{3}{2}\right)}, \quad |x|>1  \end{cases}$\\
  \textbf{6*}           & $\E^{-x^2} H_n(x)$ & $s=\frac{1}{2}$ & $2^n\frac{(-1)^{\left\lfloor \frac{n}{2}\right\rfloor} 2^{-2 \left\lfloor \frac{n}{2}\right\rfloor +n+1} x^{n-2 \left\lfloor \frac{n}{2}\right\rfloor } \Gamma \left(\left\lfloor \frac{n+1}{2}\right\rfloor +1\right) }{\sqrt{\pi }} \, _1F_1\left(\left\lfloor \frac{n+1}{2}\right\rfloor +1; n-2 \left\lfloor \frac{n}{2}\right\rfloor +\frac{1}{2};-x^2\right)$\\
    \textbf{6**}           & $\E^{-x^2} H_n(x)$ & $s=-\frac{1}{2}, n\geq1$ & $\frac{(-1)^{\left\lfloor \frac{n}{2}\right\rfloor }2^{n-1}\left\lfloor \frac{n-1}{2}\right\rfloor !\,x^{n-2 \left\lfloor \frac{n}{2}\right\rfloor }}{\Gamma \left(n-2 \left\lfloor \frac{n}{2}\right\rfloor +\frac{1}{2}\right)}\, _1F_1\left(n-\left\lfloor \frac{n}{2}\right\rfloor ;n-2 \left\lfloor \frac{n}{2}\right\rfloor +\frac{1}{2};-x^2\right)$
\end{tabular}
\captionof{table}{A listing of notable special and limit cases obtained via Tables \ref{tab:fractionallaplacetable} and \ref{tab:fractionallaplacetable2}.}
\label{tab:fractionallaplacetabl3continued}
    \end{landscape}
    \clearpage}
\restoregeometry

\section{Fractional Laplacian and Riesz potential formulae in higher dimensions}\label{sec:fractionalexplicitformshighD}
\sectionmark{Formulae in higher dimensions}
An orthogonal basis of polynomials on the $d$-dimensional unit ball may be obtained using radially shifted Jacobi polynomials $P_n^{\left(s,\frac{d+2\ell-2}{2}\right)}(2|\mathbf{x}|^2-1)$, with $\mathbf{x} \in B_1 \subset \mathbb{R}^d$. Note that without including hyperspherical harmonics, they only form a basis for radially symmetric functions. These polynomials inherit their orthogonality with respect to the weight function $w_s(x) = (1-|\mathbf{x}|^2)_+^s$ from the orthogonality of the Jacobi polynomials on the unit interval. As discussed in \cite{dunkl_orthogonal_2014}, they form a basis on the $d$-dimensional ball when multiplied by the solid harmonic polynomials. They can be considered a $d$-dimensional generalization of the widely used Zernike polynomials which have enjoyed successful applications in the natural sciences, in particular in optics applications \cite{thibos_standards_2000, roddier_atmospheric_1990,mahajan_zernike_1994,rocha_effects_2007}. As shown in \cite{dyda2017fractional}, the above weighted polynomials have known Meijer-G function representations:
\begin{align*}
\left(1-| \mathbf{x}| ^2\right)_+^a P_n^{\left(a,\frac{d+2\ell-2}{2}\right)}&\left(2 | \mathbf{x}| ^2-1\right)\\ &= \frac{(-1)^n \Gamma (a+n+1)}{n!} G_{2,2}^{1,1}\left(| \mathbf{x}| ^2 \left\lvert
\begin{array}{c}
 1-\frac{d+2\ell}{2}-n,a+n+1 \\
 0,1-\frac{d+2\ell}{2} \\
\end{array}
\right. \right)\\
&=\frac{\Gamma (a+n+1)}{n!} G_{2,2}^{2,0}\left(| \mathbf{x}| ^2 \left\lvert
\begin{array}{c}
 a+n+1, 1-\frac{d+2\ell}{2}-n \\
 1-\frac{d+2\ell}{2}, 0 \\
\end{array}
\right. \right).
\end{align*}
As a result, one may derive an explicit form of their fractional Laplacian and Riesz potentials which was first proven in \cite{dyda2017fractional} and used in \cite{gutleb2022computation} to develop a sparse spectral method for Riesz potentials on arbitrary dimensional balls. In this paper we provide the more general result to work for generic Jacobi parameters $(a,b)$. To that end, we require the following Meijer-G function representation result:
\begin{corollary}
 $$(1-|\mathbf{x}|^2)_+^a P_n^{(a,b)}(2|\mathbf{x}|^2-1) = \frac{\Gamma (a+n+1) }{n!} G_{2,2}^{2,0}\left(|\mathbf{x}|^2 \left\lvert
\begin{array}{cc}
 a+n+1, & -n-b \\
 -b, & 0 \\
\end{array}\right.
\right).$$   
\end{corollary}
\begin{proof}
This is an immediate consequence of Row 5 of \cref{tab:meijergreptable}, obtained by using the multiplicative shift formula for Meijer-G functions in \cref{eq:meijerGsym2}.
\end{proof}
A similar construction works for $d$-dimensional generalized Laguerre polynomials $L^{\alpha}_n(|x|^2)$ to obtain results on arbitrary dimensional whole space as well as certain functions on the complement of the unit ball.
\begin{theorem}\label{thm:higherdimmeijerglaplacian}
The functions $(-\Delta)^s f(x)$ in column 4 of \cref{tab:higherdimensional} are explicit forms for the fractional Laplacian in arbitrary dimensions $d \geq 1$, if $s>0$, and the Riesz potential, if $s\in(-\frac{d}{2},0)$, of the functions $f(x)$ in column 3. The $V_\ell(x)$ in \cref{tab:higherdimensional} is any $d$-dimensional solid harmonic polynomial of degree $\ell \geq 0$. Sufficient parameter conditions for the stated applications of \cref{thm:meijergDKKtheorem1,thm:meijergDKKtheorem2} are
\begin{align*}
\text{Row A:}&\quad a>-1,\qquad b<\frac{d+\ell}{2},\\
\text{Row B:}&\quad \operatorname{Re}(\alpha)>\max\left(\frac{\ell}{2}-s-n-1,-n-1\right),\\
\text{Row C:}&\quad -1<a<s-n-\frac{\ell}{2},\qquad b>-n-1+\frac{\ell}{2}-s,\\
\text{Row D:}&\quad a>-1,\qquad \max(a,a+b)<s-\frac{\ell}{2}.
\end{align*}
The formulas hold for $\mathbf{x}\neq\mathbf{0}$ and, in Rows A, C and D, with $|\mathbf{x}|\neq1$ unless $a>2s$. Values at $\mathbf{x}=\mathbf{0}$ require a separate continuity argument. \cref{tab:higherdimensionalhypergeom} contains the corresponding hypergeometric forms for $(-\Delta)^s f(x)$ as well as some additional special case results. Removable singularities in those representations are interpreted by limits.
\begin{proof}
The hypergeometric forms and special cases are obtained by hypergeometric reduction formulas like the ones listed in \cref{sec:app:reduction} -- we omit the individual proofs for them as they are analogous to the procedure in \cref{thm:fractionallaplacianformstheorem1D}. Furthermore, since the hypergeometric forms of the two functions on the complement of the unit ball are excessively long we do not include them here.
\begin{enumerate}[start=1,label={(\bfseries Row \Alph*):}]
\item Via \cref{tab:meijergreptable}, set $\mathbf{a} = (a+n+1,-b-n)$ and $\mathbf{b} = (-b,0)$ in \cref{thm:meijergDKKtheorem1,thm:meijergDKKtheorem2} with $(m,n,p,q) = (2,0,2,2)$. A closely related formula was reproduced in \cite{dyda2017eigenvalues} using the same proof approach but only for the special case $b = \frac{d+2\ell-2}{2}$.
\item Via \cref{tab:meijergreptable}, set $\mathbf{a} = (-n-\alpha)$ and $\mathbf{b} = (0,-\alpha)$ in \cref{thm:meijergDKKtheorem1,thm:meijergDKKtheorem2} with $(m,n,p,q) = (1,1,1,2)$.
\item Due to the somewhat more `exotic' nature of the functions in rows $\mathbf{C}$ and $\mathbf{D}$ in one dimension, we did not include their Meijer-G forms in the discussion in \cref{sec:meijergforms}. However, in higher dimensional applications, complements of the unit ball are of interest so we include them here. Explicit fractional Laplacian forms of Meijer-G functions on complements of unit balls were previously discussed in \cite{dyda2017fractional} but these particularly nice Jacobi polynomial special cases were not specifically described. By \cite[8.4.36.3]{prudnikovVol3}, cf.~\cite[05.06.26.0012.01]{WolframFunctions2022}, we have:
\begin{align*}
 (|\mathbf{x}|^2-1)_+^a P^{(a,b)}_n(2|\mathbf{x}|^2-1) = \frac{\Gamma (a+n+1)}{\Gamma (n+1)} G_{2,2}^{0,2}\left(|\mathbf{x}|^2\left|
\begin{array}{c}
 -b-n,a+n+1 \\
 0,-b \\
\end{array}
\right.\right).
\end{align*}
The stated result then follows by application of Theorems \ref{thm:meijergDKKtheorem1} and \ref{thm:meijergDKKtheorem2}.
\item See the discussion of row $\mathbf{C}$ above. By \cite[8.4.36.2]{prudnikovVol3} we have:
\begin{align*}
(|\mathbf{x}|^2-1)_+^a P^{(a,b)}_n\left(\frac{2}{|\mathbf{x}|^2}-1\right) = \frac{\Gamma (a+n+1)}{\Gamma (n+1)} G_{2,2}^{0,2}\left(|\mathbf{x}|^2\left|
\begin{array}{c}
 a+1,a+b+1 \\
 -n,a+b+n+1 \\
\end{array}
\right.\right).
\end{align*}
The stated result then follows by application of Theorems \ref{thm:meijergDKKtheorem1} and \ref{thm:meijergDKKtheorem2}.
\end{enumerate}
\end{proof}
\end{theorem}

    \clearpage
    \thispagestyle{empty}
    {\newgeometry{left=1cm,right=1cm,bottom=1cm,top=1cm}
    \begin{landscape}
        \centering 
        \begin{tabular}{c c c c}
 \textbf{Row ID}& \textbf{Domain}&  $\mathbf{f(x)}$ &      $\mathbf{(-\Delta)^s f(x)}$ (Meijer-G form)   \\ \hline \hline
\textbf{A} &  $B_1 \subset \mathbb{R}^d$  &$V_\ell(\mathbf{x}) \left(1-| \mathbf{x}| ^2\right)_+^a P_n^{\left(a,b\right)}\left(2 | \mathbf{x}| ^2-1\right)$ &$V_\ell(\mathbf{x}) \frac{4^s \Gamma (a+n+1)}{n!} G_{3,3}^{2,1}\left(| \mathbf{x}| ^2 \left|
\begin{array}{c}
 1-s-\frac{d+2\ell}{2},-b-n-s,a+n-s+1 \\
 0,-b-s,1-\frac{d+2\ell}{2} \\
\end{array} \right.
\right)$    \\   
\textbf{B} &  $\mathbb{R}^d$   &$V_\ell(\mathbf{x}) \E^{-|\mathbf{x}|^2} L_n^\alpha(|\mathbf{x}|^2)$ & $V_\ell(\mathbf{x})\frac{4^s}{n!} G_{2,3}^{1,2}\left(\left| \mathbf{x}\right|^2 \left|
\begin{array}{c}
 1-s-\frac{d+2 \ell}{2},-n-s-\alpha \\
 0,-s-\alpha ,1-\frac{d+2 \ell}{2} \\
\end{array}
\right.\right)$ \\
\textbf{C} &  $\mathbb{R}^d
\setminus B_1$   &$V_\ell(\mathbf{x}) (|\mathbf{x}|^2-1)_+^a P^{(a,b)}_n(2|\mathbf{x}|^2-1)$ & $V_\ell(\mathbf{x})\frac{4^s \Gamma (a+n+1)}{\Gamma (n+1)} G_{4,4}^{1,3}\left(|\mathbf{x}|^2\left|
\begin{array}{c}
 1-s-\frac{d+2\ell}{2},-b-n-s,a+n-s+1,-s \\
 0,1-\frac{d+2\ell}{2},-b-s,-s \\
\end{array}
\right.\right)$ \\
\textbf{D} &  $\mathbb{R}^d
\setminus B_1$   &$ V_\ell(\mathbf{x})(|\mathbf{x}|^2-1)_+^a P^{(a,b)}_n(\frac{2}{|\mathbf{x}|^2}-1)$ & $V_\ell(\mathbf{x}) \frac{4^s \Gamma (a+n+1)}{\Gamma (n+1)} G_{4,4}^{1,3}\left(|\mathbf{x}|^2\left|
\begin{array}{c}
 1-s-\frac{d+2 \ell}{2},a-s+1,a+b-s+1,-s \\
 0,-n-s,a+b+n-s+1,1-\frac{d+2 \ell}{2} \\
\end{array}
\right.\right)$ \\
\end{tabular}
\captionof{table}{Explicit Meijer-G function forms of order $s$ fractional Laplacians and Riesz potentials of functions depending on $\mathbf{x} \in \mathbb{R}^d$, subject to the conditions in the preceding theorem. Corresponding hypergeometric forms, including some special cases, are presented in Table \ref{tab:higherdimensionalhypergeom}.}
\label{tab:higherdimensional}
\vspace{15mm}
\begin{tabular}{c c c c}
 \textbf{Row ID}& \textbf{Domain}&  $\mathbf{f(x)}$ &      $\mathbf{(-\Delta)^s f(x)}$ (Hypergeometric form)   \\ \hline \hline
 \textbf{A} &  $B_1 \subset \mathbb{R}^d$   & $V_\ell(\mathbf{x}) \left(1-| \mathbf{x}| ^2\right)_+^a P_n^{\left(a,b\right)}\left(2 | \mathbf{x}| ^2-1\right)$ &$V_\ell(\mathbf{x})\frac{4^s \Gamma (a+n+1)}{\Gamma (n+1)}\begin{cases}\frac{\Gamma (-b-s) \Gamma \left(\frac{d}{2}+s+\ell \right) \, _3F_2\left(-a-n+s,b+n+s+1,\frac{d}{2}+s+\ell ;b+s+1,\frac{d}{2}+\ell ;|\mathbf{x}|^2\right)}{\Gamma \left(\frac{d}{2}+\ell \right) \Gamma (a+n-s+1) \Gamma (-b-n-s)}, \quad |\mathbf{x}| <1 \\ \frac{\Gamma \left(-b+\frac{d}{2}+\ell \right) \Gamma \left(\frac{d}{2}+s+\ell \right) | x| ^{-d-2 (s+\ell )} \, _3F_2\left(s+1,-b+\frac{d}{2}+\ell ,\frac{d}{2}+s+\ell ;-b+\frac{d}{2}-n+\ell ,a+\frac{d}{2}+n+\ell +1;\frac{1}{|\mathbf{x}|^2}\right)}{\Gamma (-s) \Gamma \left(a+\frac{d}{2}+n+\ell +1\right) \Gamma \left(-b+\frac{d}{2}-n+\ell \right)}, \quad |\mathbf{x}| >1 \end{cases}$  \\
\textbf{A*} &  $B_1 \subset \mathbb{R}^d$   & $V_\ell(\mathbf{x}) \left(1-| \mathbf{x}| ^2\right)_+^a P_n^{\left(a,\frac{d+2\ell-2}{2}\right)}\left(2 | \mathbf{x}| ^2-1\right)$ &$V_\ell(\mathbf{x})(-1)^n \frac{4^s \Gamma(1 + a + n)}{\Gamma(n+1)} \begin{cases} \frac{\Gamma \left(\frac{d}{2}+\ell+n+s\right) \, _2F_1\left(-a-n+s,\frac{d}{2}+\ell+n+s;\frac{d}{2}+\ell;|\mathbf{x}|^2\right)}{\Gamma \left(\frac{d}{2}+\ell\right) \Gamma (a+n-s+1)} ,\quad |\mathbf{x}|<1 \\ \frac{\Gamma \left(\frac{d}{2}+\ell+n+s\right) | x| ^{-d-2 \ell-2 n-2 s} \, _2F_1\left(n+s+1,\frac{d}{2}+\ell+n+s;a+\frac{d}{2}+\ell+2 n+1;\frac{1}{| \mathbf{x}| ^2}\right)}{\Gamma (-n-s) \Gamma \left(a+\frac{d}{2}+\ell+2 n+1\right)}, \quad  |\mathbf{x}|>1\end{cases} $  \\
\textbf{A**} &  $B_1 \subset \mathbb{R}^d$   & $V_\ell(\mathbf{x}) \left(1-| \mathbf{x}| ^2\right)_+^s P_n^{\left(s,\frac{d+2\ell-2}{2}\right)}\left(2 | \mathbf{x}| ^2-1\right)$ &$V_\ell(\mathbf{x}) \frac{4^s \Gamma(1 + s + n)}{\Gamma(n+1)} \begin{cases} \frac{\Gamma \left(\frac{d}{2}+\ell+n+s\right) P_n^{\left(s,\frac{d+2\ell-2}{2}\right)}\left(2 | \mathbf{x}| ^2-1\right)}{\Gamma \left(\frac{d}{2}+n+\ell \right)} ,\quad |\mathbf{x}|<1 \\ \frac{ (-1)^n \Gamma \left(\frac{d}{2}+\ell+n+s\right) | x| ^{-d-2 \ell-2 n-2 s} \, _2F_1\left(n+s+1,\frac{d}{2}+\ell+n+s;s+\frac{d}{2}+\ell+2 n+1;\frac{1}{| \mathbf{x}| ^2}\right)}{\Gamma (-n-s) \Gamma \left(s+\frac{d}{2}+\ell+2 n+1\right)}, \quad  |\mathbf{x}|>1\end{cases} $  \\
\textbf{A***} &  $B_1 \subset \mathbb{R}^d$     &$V_\ell(\mathbf{x})\mathbbm{1}_{B_1}$ &$V_\ell(\mathbf{x}) 4^s \begin{cases} \frac{\Gamma \left(\frac{d+2\ell}{2}+s\right) \, _2F_1\left(s,\frac{d+2\ell}{2}+s;\frac{d+2\ell}{2};| \mathbf{x}| ^2\right)}{\Gamma \left(\frac{d+2\ell}{2}\right) \Gamma (1-s)}, \quad |\mathbf{x}|<1 \\ \frac{\Gamma \left(\frac{d+2\ell}{2}+s\right) | x| ^{-d-2 s-2\ell} \, _2F_1\left(s+1,\frac{d+2\ell}{2}+s;\frac{d+2\ell+2}{2};\frac{1}{| \mathbf{x}| ^2}\right)}{\Gamma \left(\frac{d+2\ell}{2}+1\right) \Gamma (-s)}, \quad  |\mathbf{x}|>1\end{cases} $\\
\textbf{B} &  $\mathbb{R}^d$   &$V_\ell(\mathbf{x}) \E^{-|\mathbf{x}|^2} L_n^\alpha(|\mathbf{x}|^2)$ & $V_\ell(\mathbf{x})\frac{4^s  \Gamma \left(\frac{d}{2}+s+\ell \right) \Gamma (n+s+\alpha +1) \, _2F_2\left(\frac{d}{2}+s+\ell ,n+s+\alpha +1;\frac{d}{2}+\ell ,s+\alpha +1;-|\mathbf{x}|^2\right)}{\Gamma (n+1) \Gamma \left(\frac{d}{2}+\ell \right) \Gamma (s+\alpha +1)}$

\end{tabular}
\captionof{table}{Explicit hypergeometric function forms of order $s$ fractional Laplacians and Riesz potentials of functions depending on $\mathbf{x} \in \mathbb{R}^d$, subject to the conditions and joint-limit convention stated above. Formulas with row IDs containing *s are special cases of the corresponding row ID.}
\label{tab:higherdimensionalhypergeom}
    \end{landscape}}
    \clearpage
\restoregeometry

\section{A numerical method for solving $(-\Delta)^su=f$}\label{sec:numerical}

Utilizing the explicit expressions derived in this work, one may develop efficient spectral methods for solving fractional PDEs. For instance, consider the problem, find $u \in H^{1/3}(\mathbb{R}^2)$ that satisfies
\begin{align}
(-\Delta)^{1/3} u(x,y) = f(x,y) \coloneqq
\begin{cases}
20(1-r^2)^{-1/3}(x^3-3xy^2)\exp(-r^2) & \text{if} \; r^2 < 1,\\
0 & \text{otherwise},
\end{cases}
\label{eq:fpde}
\end{align}
where $r^2 = x^2+y^2$ and $H^s(\mathbb{R}^2)$, $s \in (0,1)$, denotes a fractional Sobolev space \cite{Di2012}. The right-hand side can be well approximated by weighted Zernike polynomials and the solution $u$ can be expanded in their Riesz potential counterparts. 

Generalized Zernike polynomials are orthogonal polynomials, $Z^{(b)}_{n,\ell,j}(x,y)$, supported on the unit disk orthogonal to the weight $(1-r^2)^b$. The subscript $n$ denotes the polynomial degree and $(\ell,j)$ are the Fourier mode and sign, respectively. Moreover, given an even and odd $n$ then $\ell  \in \{0,2,\dots,n\}$ and  $\ell  \in \{1,3,\dots,n\}$, respectively. Finally $j \in \{0,1\}$, unless $\ell =0$ in which case $j=1$; thus $j=0$ and $j=1$ select the sine and cosine modes, respectively, consistently with \cite{papadopoulos2023frame}. The generalized Zernike polynomials may be defined via the Jacobi polynomials as follows:
\begin{align}
Z^{(b)}_{n,\ell,j}(x,y) \coloneqq  r^\ell \sin(\ell \theta + j \pi/2) P_{(n-\ell)/2}^{(b,\ell)}(2r^2-1).
\end{align} 
Note that $r^\ell \sin(\ell \theta + j \pi/2)$ is a solid harmonic polynomial. Therefore, by setting $d=2$ and $b=s=-1/3$ in row \textbf{A**} of \cref{tab:higherdimensionalhypergeom}, we have explicit expressions for  $(-\Delta)^{-1/3} [(1-r^2)^{-1/3} Z^{(-1/3)}_{n,\ell,j}]$, denoted by $\tilde{Z}^{(s)}_{n,\ell,j}(x,y)$, which are functions supported on $\mathbb{R}^2$. 

The generic sufficient condition for Row \textbf{A} does not cover the $\ell=3$ mode in \cref{eq:fpde}; however, the specialized $G_{2,2}^{1,1}$ representation preceding \cref{thm:higherdimmeijerglaplacian} gives a direct proof. Writing $t=1/3$ for the Riesz order and $k\geq0$ for the radial Jacobi degree, the parameters for $d=2$, $\ell=3$ and $a=-1/3$ satisfy Condition S and the equal-index case of \cref{thm:meijergDKKtheorem2}, with $\bar\lambda=4+k$, $\underset{\bar{}}{\lambda}=0$ and $\nu_G=2/3$. In particular, $4+k>(\ell+2t)/2=11/6$, $0<(d+\ell)/2=5/2$, and $\nu_G>1-2t=1/3$. Thus the Row \textbf{A**} formula applies to every radial degree in the $\ell=3$ expansion used here.

More generally, the parameter restrictions stated above are sufficient conditions for the particular Meijer-G representations and proof paths used in this paper, and are not claimed to be sharp. As this example illustrates, alternative Meijer-G representations can enlarge the directly proven range. Moreover, wherever the two sides define the same analytic family in a parameter, cancellations and analytic continuation can extend the corresponding identity beyond these sufficient ranges. Establishing the maximal parameter ranges of such extensions is beyond the scope of this paper.

Note that a fast analysis (expansion) transform exists for weighted generalised Zernike polynomials \cite{Slevinsky2019, FastTransforms, olver2020fast}. Thus an $\mathcal{O}(N^2 \log N)$ complexity algorithm (where $N$ is the polynomial truncation degree) for solving \cref{eq:fpde} proceeds as follows:
\begin{enumerate}
\itemsep=0pt
\item Expand the right-hand side of \cref{eq:fpde}, $f(x,y) =  \sum\limits_{n,\ell, j} f_{n,\ell,j} (1-r^2)^{-1/3} Z^{(-1/3)}_{n,\ell,j}(x,y)$, to find the coefficients $f_{n,\ell,j}$. 
\item $u(x,y) = \sum\limits_{n,\ell, j} f_{n,\ell,j} \tilde{Z}_{n,\ell, j}(x,y)$.
\end{enumerate}
We provide a plot of the right-hand side and the solution in \cref{fig:fpde}. The code for generating the plots is available in the Julia package \texttt{FractionalFrames.jl} \cite{fractionalframes.jl} and archived on Zenodo \cite{fractionalframes-zenodo}. Although this methodology is quite sensitive to the choice of the right-hand side, techniques for adapting these explicit expressions for more general equations and data, by utilizing frame techniques \cite{Adcock2019, Adcock2020}, may be found in \cite{papadopoulos2023frame}.
\begin{figure}[h!]
\centering
\includegraphics[width =0.49 \textwidth]{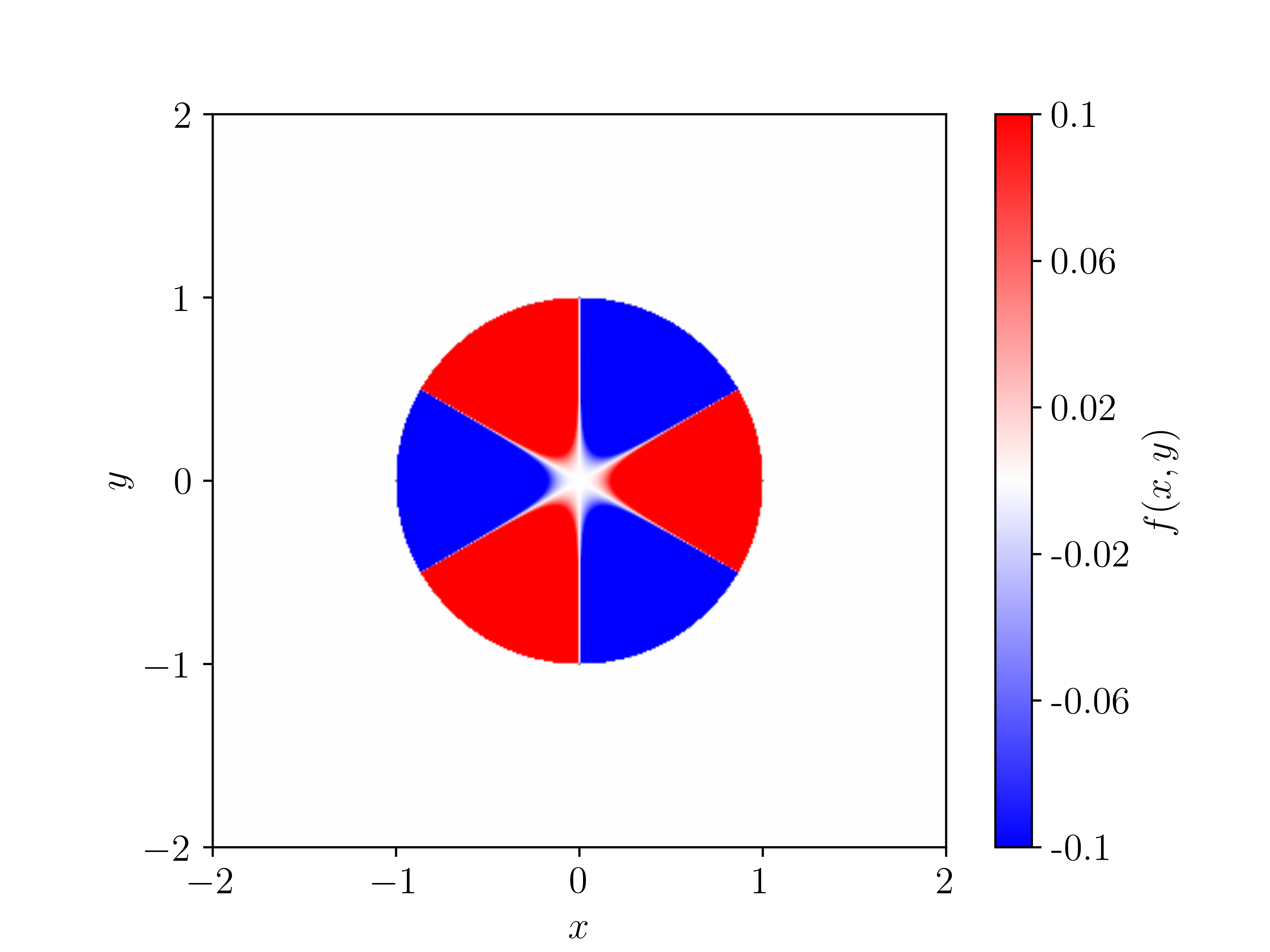}
\includegraphics[width =0.49 \textwidth]{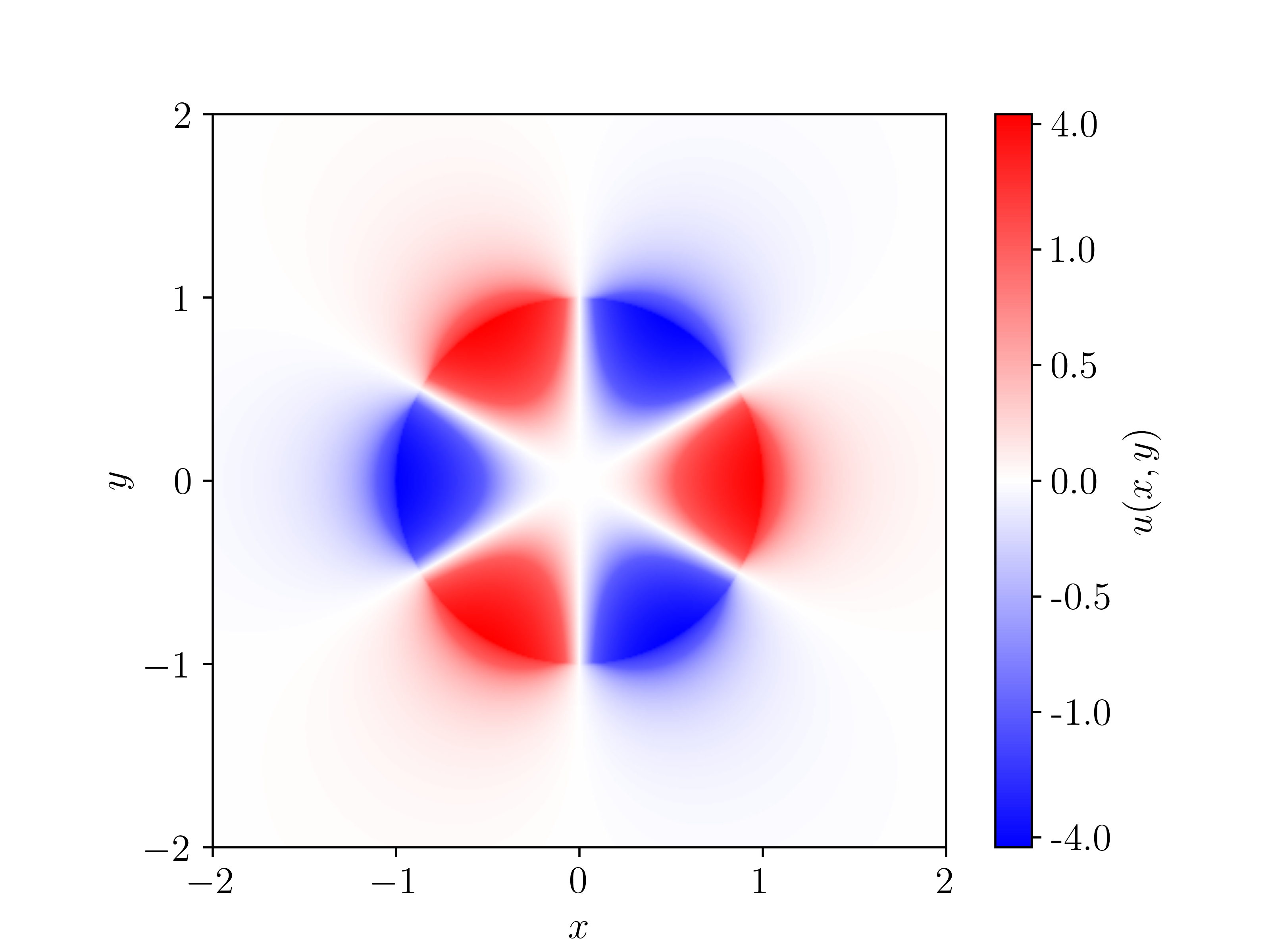}
\caption{The right-hand side $f(x,y)$ (left) and the solution $u(x,y)$ (right) of \cref{eq:fpde}. The colour bar is scaled in a logarithmic manner. The right-hand side is supported on the unit disk with a blow-up as $r \to 1$, whereas the solution is supported on $\mathbb{R}^2$.}\label{fig:fpde}
\end{figure}

\section{Conclusions}
In this paper we proved and collected numerous explicit fractional Laplacian and Riesz potential formulae for classical functions in one and higher dimensions. We expressed them both in generic Meijer-G form as well as in the more computationally approachable hypergeometric functions, based on previous results in \cite{dyda2017fractional}. By means of a simple numerical example, we have shown how these results may be used as efficient and accurate basis functions in numerical methods for equations involving the fractional Laplacian. In \cite{papadopoulos2023frame}, we have used these explicit expressions in a generically applicable whole space frame-based numerical method for time-dependent fractional PDEs involving the fractional Laplacian. Even outside of the context of spectral methods, these non-trivial expressions may be used to construct effective toy models for testing new numerical methods. Explicit expressions for Riesz potentials are also used in proofs of notable analytic results, e.g.~equilibrium measures \cite{carrillo2023radial,Jcarrilloradial}.

\section*{Acknowledgments}
This work was completed with the support of the EPSRC grant EP/T022132/1 ``Spectral element methods for fractional differential equations, with applications in applied analysis and medical imaging" and the Leverhulme Trust Research Project Grant RPG-2019-144 ``Constructive approximation theory on and inside algebraic curves and surfaces". IPAP was also supported by the Deutsche Forschungsgemeinschaft (DFG, German Research Foundation) under Germany's Excellence Strategy -- The Berlin Mathematics Research Center MATH+ (EXC-2046/1, project ID: 390685689).

TSG would like to thank Jos\'e A.~Carrillo and David Gomez-Castro for fruitful discussions on Riesz potentials, as well as Ann Brew, Elizabeth Killeen, and Myrjam M\"uhlbacher for much appreciated help in tracking down an obscure historical reference.

\newpage
\appendix
\appendixpage

\section{Proofs of the Meijer-G function forms of classical functions}
\sectionmark{Proofs of Meijer-G forms}
\label{sec:app:proofs}

\subsection*{Meijer-G representation of weighted Jacobi polynomials}
To our knowledge, the most general Meijer-G representation for weighted Jacobi polynomials, for $z \notin (-1,0)$, is \cite[8.4.36.1]{prudnikovVol3} \cite[05.06.26.0011.01]{WolframFunctions2022},
\begin{align}\label{eq:incompletecanonjacobimeijerG}
(1-z)_+^a P_n^{(a,b)}(2z-1) = \frac{\Gamma (a+n+1)}{n!} G_{2,2}^{2,0}\left(z\left|
\begin{array}{cc}
 a+n+1, & -b-n \\
 0, & -b \\
\end{array}
\right.\right).
\end{align}
We use this formula to obtain the following two results.

\begin{lemma}\label{lem:equaljacobimeijergseries}
The following weighted same-parameter Jacobi polynomials with $a>-1$ are a product of solid harmonic polynomials and Meijer-G functions:
\begin{align*}
(1-x^2)_+^a P_n^{(a,a)}(x) = \frac{\Gamma (a+n+1)}{n!} x^{n-2 \left\lfloor \frac{n}{2}\right\rfloor } G_{2,2}^{2,0}\left(x^2 \left|
\begin{array}{c}
 a+\left\lfloor \frac{n}{2}\right\rfloor +1,-n+\left\lfloor \frac{n}{2}\right\rfloor +\frac{1}{2} \\
 0,-n+2 \left\lfloor \frac{n}{2}\right\rfloor +\frac{1}{2} \\
\end{array}\right.
\right).
\end{align*}
\end{lemma}
\begin{proof}
The proof of this is based on the following lesser known property of same-parameter Jacobi polynomials \cite[18.7.13, 18.7.14]{nist_2018}:
\begin{align*}
P^{(a,a)}_n(x) = \frac{P^{(a,a)}_{n}%
\left(1\right)}{P^{(%
a,-\frac{1}{2}+n-2\left\lfloor \frac{n}{2} \right\rfloor)}_{\left\lfloor \frac{n}{2} \right\rfloor}\left(1\right)}x^{n-2\left\lfloor \frac{n}{2}\right\rfloor}P^{(a,-\frac{1}{2}+n-2\left\lfloor \frac{n}{2} \right\rfloor)}_{\left\lfloor \frac{n}{2} \right\rfloor}\left(2x^{2}-1\right).
\end{align*}
See also \cite[18.7.15, 18.7.16]{nist_2018} for a version with Gegenbauer normalization. Picking $b = -1/2+n-2\lfloor n / 2 \rfloor$ in \cref{eq:incompletecanonjacobimeijerG} and noting that
\begin{align}
\frac{\Gamma(a+\lfloor\frac{n}{2}\rfloor + 1)}{\lfloor \frac{n}{2}\rfloor!} \frac{P_n^{(a,a)}(1)}{P_{\lfloor \frac{n}{2} \rfloor}^{(a, -\frac{1}{2}+n-2\lfloor \frac{n}{2} \rfloor)}(1)} = \frac{\Gamma(a+n+1)}{n!}
\end{align}
yields the stated form. Note that $x^{n-2\left\lfloor \frac{n}{2}\right\rfloor}$ is either $1$ or $x$ depending on whether $n \in \mathbb{N}_0$ is even or odd and thus $x^{n-2\left\lfloor \frac{n}{2}\right\rfloor}$ is a solid harmonic polynomial.
\end{proof}

\begin{lemma}\label{lem:radialjacobimeijergseries}
The following weighted and radially shifted Jacobi polynomials with $a>-1$ and $b>-1/2$ are Meijer-G functions:
\begin{align*}
(1-x^2)_+^a(x^2)^b P_n^{(a,b)}(2x^2-1) = \frac{\Gamma (a+n+1) }{n!} G_{2,2}^{2,0}\left(x^2 \left\lvert
\begin{array}{cc}
 a+b+n+1 & -n \\
 0 & b \\
\end{array}\right.
\right).
\end{align*}
\end{lemma}
\begin{proof}
This follows from \eqref{eq:incompletecanonjacobimeijerG} with $z=x^2$ and the multiplicative shift property of Meijer-G functions followed by the permutation symmetry property to swap $0$ and $b$ in the bottom row of the Meijer-G function.
\end{proof}

\section{Reduction of certain Meijer-G functions to hypergeometric form}
\label{sec:app:reduction}
\sectionmark{Reduction of Meijer-G functions}
As mentioned in \cref{sec:meijergintro}, many special-case Meijer-G function reductions for explicit parameters are known and listed in formula collections such as \cite{prudnikovVol3} and \cite{WolframFunctions2022}. This section collects a few such referenced formulae in  \cref{tab:meijergreductionappendix}. Since symbolic mathematics software such as Wolfram Mathematica is largely aware of these types of reduction formulae, it is usually possible to utilize such software on very general Meijer-G function expressions to automate such reductions.

\afterpage{%
    \clearpage
    \thispagestyle{empty}
    \newgeometry{left=1cm,right=1cm,bottom=1cm,top=1cm}
    \begin{landscape}
        \centering 
\begin{tabular}{c c c}
$\mathbf{(m,n,p,q)}$ & \textbf{Hypergeometric reduction} & \textbf{Reference}  \\ \hline \hline
(1,1,2,2) & $\begin{cases}\frac{z^{b_1} \Gamma \left(-a_1+b_1+1\right) \, _2F_1\left(-a_1+b_1+1,-a_2+b_1+1;b_1-b_2+1;z\right)}{\Gamma \left(b_1-b_2+1\right) \Gamma \left(a_2-b_1\right)},\quad |z|<1,\\ \frac{z^{a_1-1} \Gamma \left(-a_1+b_1+1\right) \, _2F_1\left(-a_1+b_1+1,-a_1+b_2+1;-a_1+a_2+1;\frac{1}{z}\right)}{\Gamma \left(-a_1+a_2+1\right) \Gamma \left(a_1-b_2\right)}, \quad |z|>1\end{cases}$ & \cite[07.34.03.0321.01]{WolframFunctions2022}\\
(1,2,2,3) & $\frac{z^{b_1} \Gamma \left(-a_1+b_1+1\right) \Gamma \left(-a_2+b_1+1\right) \, _2F_2\left(-a_1+b_1+1,-a_2+b_1+1;b_1-b_2+1,b_1-b_3+1;-z\right)}{\Gamma \left(b_1-b_2+1\right) \Gamma \left(b_1-b_3+1\right)}$ & \cite[07.34.03.0471.01]{WolframFunctions2022}\\
 (2,1,2,3) & $\frac{z^{b_1} \Gamma \left(-a_1+b_1+1\right) \, _2F_2\left(-a_1+b_1+1,-a_2+b_1+1;b_1-b_2+1,b_1-b_3+1;-z\right)}{\pi^{-1}\sin \left(\pi  \left(b_2-b_1\right)\right)\Gamma \left(b_1-b_2+1\right) \Gamma \left(b_1-b_3+1\right) \Gamma \left(a_2-b_1\right)}-\frac{z^{b_2} \Gamma \left(-a_1+b_2+1\right) \, _2F_2\left(-a_1+b_2+1,-a_2+b_2+1;-b_1+b_2+1,b_2-b_3+1;-z\right)}{\pi^{-1}\sin \left(\pi  \left(b_2-b_1\right)\right)\Gamma \left(-b_1+b_2+1\right) \Gamma \left(b_2-b_3+1\right) \Gamma \left(a_2-b_2\right)}$ & \cite[07.34.03.0831.01]{WolframFunctions2022}\\
  (2,1,2,4) & $\frac{z^{b_1} \Gamma \left(-a_1+b_1+1\right) \, _2F_3\left(-a_1+b_1+1,-a_2+b_1+1;b_1-b_2+1,b_1-b_3+1,b_1-b_4+1;-z\right)}{\pi^{-1}\sin \left(\pi  \left(b_2-b_1\right)\right) \Gamma \left(b_1-b_2+1\right) \Gamma \left(b_1-b_3+1\right) \Gamma \left(b_1-b_4+1\right) \Gamma \left(a_2-b_1\right)}-\frac{z^{b_2} \Gamma \left(-a_1+b_2+1\right) \, _2F_3\left(-a_1+b_2+1,-a_2+b_2+1;-b_1+b_2+1,b_2-b_3+1,b_2-b_4+1;-z\right)}{\pi^{-1}\sin \left(\pi  \left(b_2-b_1\right)\right) \Gamma \left(-b_1+b_2+1\right) \Gamma \left(b_2-b_3+1\right) \Gamma \left(b_2-b_4+1\right) \Gamma \left(a_2-b_2\right)}$ & \cite[07.34.03.0833.01]{WolframFunctions2022}\\
 (2,1,3,3)       & $\begin{cases} \frac{  \frac{z^{b_1} \Gamma \left(-a_1+b_1+1\right) \, _3F_2\left(-a_1+b_1+1,-a_2+b_1+1,-a_3+b_1+1;b_1-b_2+1,b_1-b_3+1;z\right)}{\Gamma \left(b_1-b_2+1\right) \Gamma \left(b_1-b_3+1\right) \Gamma \left(a_2-b_1\right) \Gamma \left(a_3-b_1\right)}-\frac{z^{b_2} \Gamma \left(-a_1+b_2+1\right) \, _3F_2\left(-a_1+b_2+1,-a_2+b_2+1,-a_3+b_2+1;-b_1+b_2+1,b_2-b_3+1;z\right)}{\Gamma \left(-b_1+b_2+1\right) \Gamma \left(b_2-b_3+1\right) \Gamma \left(a_2-b_2\right) \Gamma \left(a_3-b_2\right)}}{\pi^{-1} \sin \left(\pi  \left(b_2-b_1\right)\right)}, \quad |z|<1 \\ \tfrac{z^{a_1-1} \Gamma \left(-a_1+b_1+1\right) \Gamma \left(-a_1+b_2+1\right)  \, _3F_2\left(-a_1+b_1+1,-a_1+b_2+1,-a_1+b_3+1;-a_1+a_2+1,-a_1+a_3+1;\frac{1}{z}\right)}{\Gamma \left(-a_1+a_2+1\right) \Gamma \left(-a_1+a_3+1\right) \Gamma \left(a_1-b_3\right)} \ ,\quad |z|>1\end{cases}$    & \cite[07.34.03.0850.01]{WolframFunctions2022}
\end{tabular}
\captionof{table}{Reproduction of some hypergeometric reduction formulae for special cases of the Meijer-G function $G^{m,n}_{p,q} \left(z\left|
\begin{array}{c}
 \mathbf{a} \\
 \mathbf{b} \\
\end{array}\right.
\right)$. The formulas are valid wherever the stated reduction formula is well defined, e.g.~whenever $b_2-b_1 \notin \mathbb{Z}$ for the third, fourth and fifth rows. See the references for the explicit conditions for each row.}
\label{tab:meijergreductionappendix}
    \end{landscape}
    \clearpage
}
\restoregeometry

\printbibliography

@article{Adcock2019,
  title={Frames and numerical approximation},
  author={Adcock, Ben and Huybrechs, Daan},
  journal={SIAM Review},
  volume={61},
  number={3},
  pages={443--473},
  year={2019},
  publisher={SIAM},
  doi={10.1137/17M1114697}
}

@article{pagnini2023mellin,
  title={Mellin definition of the fractional {L}aplacian},
  author={Pagnini, Gianni and Runfola, Claudio},
  journal={Fractional Calculus and Applied Analysis},
  pages={1--17},
  year={2023},
  publisher={Springer},
  doi={10.1007/s13540-023-00190-z}
}

@article{beals2013meijer,
  title={Meijer {G}-{Functions}: {A} {Gentle} {Introduction}},
  author={Beals, Richard and Szmigielski, Jacek},
  journal={Notices of the AMS},
  volume={60},
  number={7},
  pages={866--872},
  year={2013}
}

@article{carrillo2023radial,
  title={From radial symmetry to fractal behavior of aggregation equilibria for repulsive--attractive potentials},
  author={Carrillo, Jos{\'e} A and Shu, Ruiwen},
  journal={Calculus of Variations and Partial Differential Equations},
  volume={62},
  number={1},
  pages={28},
  year={2023},
  publisher={Springer},
  doi={10.1007/s00526-022-02368-4}
}

@article{Jcarrilloradial,
title = {Explicit equilibrium solutions for the aggregation equation with power-law potentials},
journal = {Kinetic and Related Models},
volume = {10},number = {1},pages = {171-192},
year = {2017},
issn = {1937-5093},
doi = {10.3934/krm.2017007},
author = {José A. Carrillo and Yanghong Huang}
}

@article{gutleb2023polynomial,
  title={Polynomial and rational measure modifications of orthogonal polynomials via infinite-dimensional banded matrix factorizations},
  author={Gutleb, Timon S and Olver, Sheehan and Slevinsky, Richard Mikael},
  journal={Foundations of Computational Mathematics},
  pages={1--43},
  year={2024},
  publisher={Springer},
  doi={10.1007/s10208-024-09671-w}
}

@inproceedings{thibos_standards_2000,
	address = {Santa Fe, New Mexico},
	title = {Standards for Reporting the Optical Aberrations of Eyes},
	doi = {10.1364/VSIA.2000.SuC1},
	booktitle = {Vision {Science} and its {Applications}},
	publisher = {OSA},
	author = {Thibos, L. N. and Applegate, R. A. and Schwiegerling, J. T. and Webb, R. and {VSIA Standards Taskforce Members}},
	year = {2000}
}

@article{roddier_atmospheric_1990,
  author   = {Roddier, N. A.},
  journal  = {Optical Engineering},
  title    = {Atmospheric wavefront simulation using {Zernike} polynomials},
  year     = {1990},
  number   = {10},
  pages    = {1174--1180},
  volume   = {29},
  doi      = {10.1117/12.55712},
  publisher = {SPIE}
}

@Article{mahajan_zernike_1994,
  author  = {Mahajan, V. N.},
  journal = {Applied Optics},
  title   = {Zernike circle polynomials and optical aberrations of systems with circular pupils},
  year    = {1994},
  number  = {34},
  pages   = {8121},
  volume  = {33},
  doi     = {10.1364/AO.33.008121},
}

@Article{rocha_effects_2007,
  author  = {Rocha, K. M. and Vabre, L. and Harms, F. and Chateau, N. and Krueger, R. R.},
  journal = {Journal of Refractive Surgery},
  title   = {Effects of {Zernike} wavefront aberrations on visual acuity measured using electromagnetic adaptive optics technology},
  year    = {2007},
  number  = {9},
  pages   = {953--959},
  volume  = {23},
  doi     = {10.3928/1081-597X-20071101-17},
}

@article{olver2020fast,
  title={Fast algorithms using orthogonal polynomials},
  author={Olver, Sheehan and Slevinsky, Richard Mika{\"e}l and Townsend, Alex},
  journal={Acta Numerica},
  volume={29},
  pages={573--699},
  year={2020},
  publisher={Cambridge University Press},
  doi={10.1017/S0962492920000045}
}

@book{dunkl_orthogonal_2014,
	address = {Cambridge},
	edition = {Second edition},
	series = {Encyclopedia of mathematics and its applications},
	title = {Orthogonal polynomials of several variables},
	isbn = {9781107071896},
	number = {155},
	publisher = {Cambridge University Press},
	author = {Dunkl, C. F. and Xu, Y.},
	year = {2014},
}

@article{kravchenko2017representation,
  title={Representation of solutions to the one-dimensional {Schr{\"o}dinger} equation in terms of {Neumann} series of {B}essel functions},
  author={Kravchenko, Vladislav V and Navarro, Luis J and Torba, Sergii M},
  journal={Applied Mathematics and Computation},
  volume={314},
  pages={173--192},
  year={2017},
  publisher={Elsevier},
  doi={10.1016/j.amc.2017.07.006}
}

@article{jiang2008efficient,
  title={Efficient representation of nonreflecting boundary conditions for the time-dependent {Schr{\"o}dinger} equation in two dimensions},
  author={Jiang, Shidong and Greengard, Leslie},
  journal={Communications on Pure and Applied Mathematics: A Journal Issued by the Courant Institute of Mathematical Sciences},
  volume={61},
  number={2},
  pages={261--288},
  year={2008},
  publisher={Wiley Online Library},
  doi={10.1002/cpa.20200}
}

@book{korenev2002bessel,
  title={Bessel functions and their applications},
  author={Korenev, Boris Grigorʹevich},
  year={2002},
  publisher={CRC Press},
  isbn={978-0367454852}
}

@article{sasaki2016one,
  title={One-dimensional {Schr{\"o}dinger} equation with non-analytic potential $V(x)=-g^2\exp(-|x|)$ and its exact {Bessel}-function solvability},
  author={Sasaki, Ryu and Znojil, Miloslav},
  journal={Journal of Physics A: Mathematical and Theoretical},
  volume={49},
  number={44},
  pages={445303},
  year={2016},
  publisher={IOP Publishing},
  doi={10.1088/1751-8113/49/44/445303}
}

@article{okada2006theory,
  title={Theory of efficient array observations of microtremors with special reference to the {SPAC} method},
  author={Okada, Hiroshi},
  journal={Exploration Geophysics},
  volume={37},
  number={1},
  pages={73--85},
  year={2006},
  publisher={Taylor \& Francis},
  doi={10.1071/EG06073}
}

@article{chavez2005alternative,
  title={An alternative approach to the {SPAC} analysis of microtremors: exploiting stationarity of noise},
  author={Ch{\'a}vez-Garc{\'\i}a, Francisco J and Rodr{\'\i}guez, Miguel and Stephenson, William R},
  journal={Bulletin of the Seismological Society of America},
  volume={95},
  number={1},
  pages={277--293},
  year={2005},
  publisher={Seismological Society of America},
  doi={10.1785/0120030179}
}

@article{kittel1968x,
  title={{X-ray diffraction from helices: Structure of DNA}},
  author={Kittel, C.},
  journal={American Journal of Physics},
  volume={36},
  number={7},
  pages={610--616},
  year={1968},
  publisher={American Association of Physics Teachers},
  doi={10.1119/1.1975029}
}

@article{garbaczewski2019fractional,
  title={{Fractional Laplacians in bounded domains: Killed, reflected, censored, and taboo L{\'e}vy flights}},
  author={Garbaczewski, Piotr and Stephanovich, Vladimir},
  journal={Physical Review E},
  volume={99},
  number={4},
  pages={042126},
  year={2019},
  publisher={APS},
  doi={10.1103/PhysRevE.99.042126}
}

@article{townsend2018fast,
  title={{Fast polynomial transforms based on Toeplitz and Hankel matrices}},
  author={Townsend, Alex and Webb, Marcus and Olver, Sheehan},
  journal={Mathematics of Computation},
  volume={87},
  number={312},
  pages={1913--1934},
  year={2018},
  doi={10.1090/mcom/3277}
}

@article{bateman1909solution,
  title={The solution of linear differential equations by means of definite integrals},
  author={Bateman, Harry},
  journal={Transactions of the Cambridge Philosophical Society},
  volume={21},
  pages={171--196},
  year={1909}
}

@article{chen2016generalized,
  title={Generalized {Jacobi} functions and their applications to fractional differential equations},
  author={Chen, Sheng and Shen, Jie and Wang, Li-Lian},
  journal={Mathematics of Computation},
  volume={85},
  number={300},
  pages={1603--1638},
  year={2016},
  doi={10.1090/mcom3035}
}

@article{kilbas2016generalized,
  title={{The generalized hypergeometric function as the Meijer G-function}},
  author={Kilbas, Anatoly A and Saxena, Ram K and Saigo, Megumi and Trujillo, Juan J},
  journal={Analysis},
  volume={36},
  number={1},
  pages={1--14},
  year={2016},
  publisher={Oldenbourg Wissenschaftsverlag},
  doi={10.1515/anly-2015-5001}
}

@book{bateman1953higher,
  title={Higher transcendental functions},
  author={Bateman, Harry},
  volume={1},
  year={1953},
  publisher={McGraw-Hill Book Company}
}

@book{paris2001asymptotics,
  title={Asymptotics and {M}ellin--{B}arnes integrals},
  author={Paris, Richard B and Kaminski, David},
  volume={85},
  year={2001},
  publisher={Cambridge University Press},
  isbn={9780511546662},
  doi={10.1017/CBO9780511546662}
}

@book{prudnikovVol3,
  title={{Integrals} and {Series}, {Volume} 3: {More} {Special} {Functions}},
  author={Prudnikov, A. P. and Brychkov, Yu. A. and Marichev, O. I.},
  year={1990},
  publisher={Gordon and Breach Science Publishers},
  location={New York},
  isbn={2881246826}
}

@article{fall2021morse,
  title={Morse index versus radial symmetry for fractional {D}irichlet problems},
  author={Fall, Mouhamed Moustapha and Feulefack, Pierre Aime and Temgoua, Remi Yvant and Weth, Tobias},
  journal={Advances in Mathematics},
  volume={384},
  pages={107728},
  year={2021},
  publisher={Elsevier},
  doi={10.1016/j.aim.2021.107728}
}

@article{banuelos2004cauchy,
  title={{The Cauchy process and the Steklov problem}},
  author={Ba{\~n}uelos, Rodrigo and Kulczycki, Tadeusz},
  journal={Journal of Functional Analysis},
  volume={211},
  number={2},
  pages={355--423},
  year={2004},
  publisher={Elsevier},
  doi={10.1016/j.jfa.2004.02.005}
}

@article{dyda2017eigenvalues,
  title={Eigenvalues of the fractional {L}aplace operator in the unit ball},
  author={Dyda, Bart{\l}omiej and Kuznetsov, Alexey and Kwa{\'s}nicki, Mateusz},
  journal={Journal of the London Mathematical Society},
  volume={95},
  number={2},
  pages={500--518},
  year={2017},
  publisher={Wiley Online Library},
  doi={10.1112/jlms.12024}
}

@book{oldham2009atlas,
  title={An atlas of functions: with equator, the atlas function calculator},
  author={Oldham, Keith B and Myland, Jan and Spanier, Jerome},
  year={2009},
  publisher={Springer},
  isbn={978-0-387-48806-6},
  doi={10.1007/978-0-387-48807-3}
}

@article{Adcock2020,
  title={Frames and numerical approximation {II}: generalized sampling},
  author={Adcock, Ben and Huybrechs, Daan},
  journal={Journal of Fourier Analysis and Applications},
  volume={26},
  number={6},
  pages={1--34},
  year={2020},
  publisher={Springer},
  doi={10.1007/s00041-020-09796-w}
}

@article{biler2011barenblatt,
  title={Barenblatt profiles for a nonlocal porous medium equation},
  author={Biler, Piotr and Imbert, Cyril and Karch, Grzegorz},
  journal={Comptes Rendus Mathematique},
  volume={349},
  number={11-12},
  pages={641--645},
  year={2011},
  publisher={Elsevier},
  doi={10.1016/j.crma.2011.06.003}
}

@article{biler2015nonlocal,
  title={The nonlocal porous medium equation: Barenblatt profiles and other weak solutions},
  author={Biler, Piotr and Imbert, Cyril and Karch, Grzegorz},
  journal={Archive for Rational Mechanics and Analysis},
  volume={215},
  number={2},
  pages={497--529},
  year={2015},
  publisher={Springer},
  doi={10.1007/s00205-014-0786-1}
}

@book{landkof1972foundations,
  title={Foundations of modern potential theory},
  author={Landkof, Naum Samo{\u\i}lovich},
  volume={180},
  year={1972},
  publisher={Springer},
  isbn={978-3-642-65185-4}
}

@article{huang2014explicit,
  title={{Explicit Barenblatt profiles for fractional porous medium equations}},
  author={Huang, Yanghong},
  journal={Bulletin of the London Mathematical Society},
  volume={46},
  number={4},
  pages={857--869},
  year={2014},
  publisher={Oxford University Press},
  doi={10.1112/blms/bdu045}
}

@article{gutleb2020computing,
  title={Computing equilibrium measures with power law kernels},
  author={Gutleb, T. S. and Carrillo, J. A. and Olver, S.},
  journal={Mathematics of Computation},
  year={2021},
  publisher={American Mathematical Society},
  doi={10.1090/mcom/3740}
}

@article{dyda2017fractional,
  title={{Fractional Laplace operator and Meijer G-function}},
  author={Dyda, Bart{\l}omiej and Kuznetsov, Alexey and Kwa{\'s}nicki, Mateusz},
  journal={Constructive Approximation},
  volume={45},
  number={3},
  pages={427--448},
  year={2017},
  publisher={Springer},
  doi={10.1007/s00365-016-9336-4}
}

@article{papadopoulos2022sparse,
  title={A sparse spectral method for fractional differential equations in one-spatial dimension},
  author={Papadopoulos, Ioannis PA and Olver, Sheehan},
  journal={Advances in Computational Mathematics},
  volume={50},
  number={4},
  pages={69},
  year={2024},
  publisher={Springer},
  doi={10.1007/s10444-024-10164-1}
}

@article{gutleb2022computation,
  title={Computation of power law equilibrium measures on balls of arbitrary dimension},
  author={Gutleb, T. S. and Carrillo, J. A. and Olver, S.},
  journal={Constructive Approximation},
  pages={1--46},
  year={2022},
  publisher={Springer},
  doi={10.1007/s00365-022-09606-0}
}

@misc{nist_2018,
	author = {{F.W.J.} Olver and {A.B.O.} Daalhuis and {D.W.} Lozier and {B.I.} Schneider and {R.F.} Boisvert and {C.W.} Clark and {B.R.} Miller and {B. V.} Saunders (eds.)},
	month = dec,
	note = {http://dlmf.nist.gov},
	shorttitle = {{NIST} {DigitalMath} {Lib}},
	title = {{NIST} {Digital} {Library} of {Mathematical} {Functions}},
	volume = {Version 1.1.8},
	year = {2022}}

@book{King2009a, 
	place={Cambridge}, 
	series={Encyclopedia of Mathematics and its Applications}, 
	title={Hilbert {T}ransforms: {V}olume 1}, 
	volume={1}, 
	doi={10.1017/CBO9780511721458}, 
	publisher={Cambridge University Press}, 
	author={King, Frederick W.}, 
	year={2009}, 
	collection={Encyclopedia of Mathematics and its Applications}
}

@article{Kwasnicki2017,
  title={Ten equivalent definitions of the fractional {L}aplace operator},
  author={Kwa{\'s}nicki, Mateusz},
  journal={Fractional Calculus and Applied Analysis},
  volume={20},
  number={1},
  pages={7--51},
  year={2017},
  publisher={De Gruyter},
  doi={10.1515/fca-2017-0002}
}

@article{Di2012,
  title={Hitchhiker's guide to the fractional {S}obolev spaces},
  author={Di Nezza, Eleonora and Palatucci, Giampiero and Valdinoci, Enrico},
  journal={Bulletin des Sciences Math{\'e}matiques},
  volume={136},
  number={5},
  pages={521--573},
  year={2012},
  publisher={Elsevier},
  doi={10.1016/j.bulsci.2011.12.004}
}

@article{Li2021,
  title={Efficient {H}ermite Spectral-{G}alerkin Methods for Nonlocal Diffusion Equations in Unbounded Domains},
  author={Li, Huiyuan and Liu, Ruiqing and Wang, Li-Lian},
  journal={Numerical Mathematics: Theory, Methods and Applications},
  year={2022},
  doi={10.4208/nmtma.OA-2022-0007s}
}

@misc{WolframFunctions2022,
	author = {{Wolfram Research, Inc.}},
	title = {{The} {Mathematical} {Functions} {Site}},
	url = {functions.wolfram.com},
	urldate = {2022-09-21},
	year = {2022}}

@article{riesz1938integrales,
  title={Int{\'e}grales de {Riemann-Liouville} et potentiels},
  author={Riesz, Marcel},
  journal={Acta Scientiarum Mathematicarum, Acta Universitatis Szegediensis},
  volume={9},
  pages={1--42},
  year={1938}
}

@article{riesz1938rectification,
  title={Rectification au travail {'Int{\'e}grales de Riemann-Liouville et potentiels'}},
  author={Riesz, Marcel},
  journal={Acta Scientiarum Mathematicarum, Acta Universitatis Szegediensis},
  year={1938}
}

@article{blumenthal1961distribution,
  title={On the distribution of first hits for the symmetric stable processes},
  author={Blumenthal, Robert M and Getoor, Ronald K and Ray, DB},
  journal={Transactions of the American Mathematical Society},
  volume={99},
  number={3},
  pages={540--554},
  year={1961},
  doi={10.2307/1993561}
}

@inproceedings{hmissi1994fonctions,
  title={Fonctions harmoniques pour les potentiels de {Riesz} sur la boule unit{\'e}},
  author={Hmissi, Farida},
  booktitle={Expositiones Mathematicae},
  volume={12},
  number={3},
  pages={281--288},
  year={1994}
}

@article{bogdan1999representation,
  title={Representation of $\alpha$-harmonic functions in {Lipschitz} domains},
  author={Bogdan, Krzysztof},
  journal={Hiroshima Mathematical Journal},
  volume={29},
  number={2},
  pages={227--243},
  year={1999},
  publisher={Hiroshima University},
  doi={10.32917/hmj/1206125005}
}

@book{rubin1996fractional,
  title={Fractional integrals and potentials},
  author={Rubin, Boris},
  volume={82},
  year={1996},
  publisher={CRC Press},
  isbn={978-0582253414}
}

@Article{treeby2010k,
  author    = {Treeby, Bradley E and Cox, Benjamin T},
  title     = {$k$-Wave: {MATLAB} toolbox for the simulation and reconstruction of photoacoustic wave fields},
  journal   = {Journal of Biomedical Optics},
  year      = {2010},
  volume    = {15},
  number    = {2},
  pages     = {021314--021314},
  publisher = {Society of Photo-Optical Instrumentation Engineers},
  doi={10.1117/1.3360308}
}

@article{treeby2012modeling,
  title={Modeling nonlinear ultrasound propagation in heterogeneous media with power law absorption using a k-space pseudospectral method},
  author={Treeby, Bradley E and Jaros, Jiri and Rendell, Alistair P and Cox, BT},
  journal={The Journal of the Acoustical Society of America},
  volume={131},
  number={6},
  pages={4324--4336},
  year={2012},
  publisher={Acoustical Society of America},
  doi={10.1121/1.4712021}
}

@inproceedings{treeby2014modelling,
  title={Modelling elastic wave propagation using the k-{W}ave {MATLAB} toolbox},
  author={Treeby, Bradley E and Jaros, Jiri and Rohrbach, Daniel and Cox, BT},
  booktitle={2014 IEEE international ultrasonics symposium},
  pages={146--149},
  year={2014},
  organization={IEEE},
  doi={10.1109/ULTSYM.2014.0037}
}

@article{treeby2018rapid,
  title={Rapid calculation of acoustic fields from arbitrary continuous-wave sources},
  author={Treeby, Bradley E and Budisky, Jakub and Wise, Elliott S and Jaros, Jiri and Cox, BT},
  journal={The Journal of the Acoustical Society of America},
  volume={143},
  number={1},
  pages={529--537},
  year={2018},
  publisher={Acoustical Society of America},
  doi={10.1121/1.5021245}
}

@article{abatangelo_fractional_2021,
	title = {Fractional {Laplacians} on ellipsoids},
	volume = {3},
	issn = {2640-3501},
	doi = {10.3934/mine.2021038},
	number = {5},
	journal = {Mathematics in Engineering},
	author = {Abatangelo, Nicola and Jarohs, Sven and Saldaña, Alberto},
	year = {2021},
	pages = {1--34},
}

@misc{papadopoulos2023frame,
  title={A frame approach for equations involving the fractional {L}aplacian}, 
  author={Papadopoulos, Ioannis P. A. and Gutleb, Timon S. and Carrillo, Jos{\'e} A. and Olver, Sheehan},
  year={2023},
  eprint={2311.12451},
  archivePrefix={arXiv},
  primaryClass={math.NA}
}

@misc{fractionalframes.jl,
title={{FractionalFrames.jl}},
year={2023},
version = {0.0.1},
url={https://github.com/ioannisPApapadopoulos/FractionalFrames.jl},
}

@article{Slevinsky2019,
  title={Fast and backward stable transforms between spherical harmonic expansions and bivariate {F}ourier series},
  author={Slevinsky, Richard Mika{\"e}l},
  journal={Applied and Computational Harmonic Analysis},
  volume={47},
  number={3},
  pages={585--606},
  year={2019},
  publisher={Elsevier},
  doi={10.1016/j.acha.2017.11.001}
}

@misc{FastTransforms,
title={Fast{T}ransforms},
author={Slevinsky, Richard Mika{\"e}l},
year={2018},
url={https://github.com/MikaelSlevinsky/FastTransforms}
}

@software{fractionalframes-zenodo,
  author       = {Papadopoulos, Ioannis P.~A.},
  title        = {{ioannisPApapadopoulos/FractionalFrames.jl}},
  month        = nov,
  year         = 2023,
  publisher    = {Zenodo},
  version      = {0.0.1},
  doi          = {10.5281/zenodo.10151679}
}

@article{king_modelling_2024,
  title={Modelling power-law ultrasound absorption using a time-fractional, static memory, {F}ourier pseudo-spectral method},
  author={King, Matthew J and Gutleb, Timon S and Treeby, BE and Cox, BT},
  journal={The Journal of the Acoustical Society of America},
  volume={157},
  number={3},
  pages={1761--1771},
  year={2025},
  publisher={AIP Publishing},
  doi={10.1121/10.0035937}
}

@article{gutleb2023static,
  title={A static memory sparse spectral method for time-fractional {PDE}s},
  author={Gutleb, Timon S and Carrillo, Jos{\'e} A},
  journal={Journal of Computational Physics},
  volume={494},
  pages={112522},
  year={2023},
  publisher={Elsevier},
  doi={10.1016/j.jcp.2023.112522}
}

@software{meijergjl-zenodo,
  author       = {Gutleb, Timon S.},
  title        = {{MeijerG.jl}},
  year         = 2026,
  publisher    = {Zenodo},
  doi          = {10.5281/zenodo.19430427}
}

@software{explicitfractionallaplacianrieszexamples-zenodo,
  author       = {Gutleb, Timon S.},
  title        = {{ExplicitFractionalLaplacianRieszExamples.jl}},
  year         = 2026,
  publisher    = {Zenodo},
  version      = {0.1},
  doi          = {10.5281/zenodo.21360098}
}
\end{document}